\documentclass[9pt]{amsart}
\usepackage{amsthm, amsfonts, amssymb, amsmath, amscd, graphicx, enumitem}
\usepackage[all]{xy}
\usepackage[margin=4cm]{geometry}

\newcommand{\into}{\hookrightarrow}

\newcommand{\Z}{\mathbb{Z}}

\newcommand{\F}{F}
\newcommand{\K}{K}

\newcommand{\e}{\varepsilon}

\newcommand{\la}{\langle}
\newcommand{\ra}{\rangle}

\newcommand{\Nm}{\mathrm{N}}
\newcommand{\UM}{\mathrm{Um}}
\newcommand{\Tr}{\mathrm{Tr}}

\newcommand{\res}{\mathrm{res}}
\newcommand{\ind}{\mathrm{ind}}
\newcommand{\iso}{\mathrm{iso}}

\newcommand{\bl}{\bullet}
\newcommand{\wh}{\circ}

\newtheorem{thm}{Theorem}[section]

\newtheorem{lem}[thm]{Lemma}
\newtheorem{cor}[thm]{Corollary}
\newtheorem{prop}[thm]{Proposition}

\theoremstyle{definition}
\newtheorem{defn}[thm]{Definition}
\newtheorem{rem}[thm]{Remark}
\newtheorem{ex}[thm]{Example}

\begin{document}
\title{Unimodular graphs and Eisenstein sums} 
\author{Bogdan Nica}
\date{\today}
\subjclass[2010]{05C50, 05C25, 11T24}
\address{Mathematisches Institut, Georg-August Universit\"at G\"ottingen}
\email{bogdan.nica@gmail.com}

\begin{abstract} Motivated in part by combinatorial applications to certain sum-product phenomena, we introduce unimodular graphs over finite fields and, more generally, over finite valuation rings. We compute the spectrum of the unimodular graphs, by using Eisenstein sums associated to unramified extensions of such rings. We derive an estimate for the number of solutions to the restricted dot product equation $a\cdot b=r$ over a finite valuation ring. Furthermore, our spectral analysis leads to the exact value of the isoperimetric constant for half of the unimodular graphs. We also compute the spectrum of Platonic graphs over finite valuation rings, and products of such rings - e.g., $\Z/(N)$. In particular, we deduce an improved lower bound for the isoperimetric constant of the Platonic graph over $\Z/(N)$.
\end{abstract}
\maketitle

\section{Introduction}
\subsection{Unimodular graphs over finite fields} This paper is concerned with adjacency spectra of certain finite graphs. One reason for being interested in such spectral computations is that they provide interesting combinatorial consequences. So, by way of motivation, let us start with an application to sum-product phenomena in finite fields that can be approached in this way.  

 Let $\F$ be a field with $q$ elements, where $q$ is a power of a prime. Throughout, we assume that $q$ is odd. Consider the $n$-dimensional space $\F^n$, where $n\geq 2$, and endow it with the dot product $a\cdot b=a_1b_1+\cdots+a_nb_n$. The problem we want to address is that of estimating the number of solutions for the equation $a\cdot b=r$, namely
\begin{align*}
N_r(A,B)=\big|\{(a,b)\in A\times B: a\cdot b=r\}\big|,
\end{align*}
for given $r\in F$ and non-empty $A,B\subseteq \F^n$. The expected value for $N_r(A,B)$ is $q^{-1}|A||B|$, so what we are aiming for is rather an upper bound for the deviation from the expected value. Once we have such control, we may derive in particular sufficient conditions that guarantee $N_r(A,B)>0$, that is, $r\in A\cdot B=\{a\cdot b: a\in A,b\in B\}$. We find it preferable to formulate such conditions in a relative, normalized way rather than in absolute terms. Namely, we let $\delta(A)=|A|/|F^n|$ denote the density of a subset $A\subseteq F^n$.

\begin{thm}\label{thm: counting field}
We have
\begin{align*}
\big|N_1(A,B) - q^{-1}|A||B|\big|< q^{(n-1)/2} \sqrt{|A||B|}.
\end{align*}
In particular, $1\in A\cdot B$ whenever $\delta(A)\: \delta(B)\geq q^{-(n-1)}$.
\end{thm}
By scaling $A$ or $B$, we immediately see that the same holds for each non-zero $r\in F$ in place of $1$. When $r=0$, we may apply the above bound to $A\times \{1\}, B\times \{1\}\subseteq F^{n+1}$. We thus get the following.
\begin{cor}\label{cor: general counting field}
For each $r\in F$ we have
\begin{align*}
\big|N_r(A,B) - q^{-1}|A||B|\big|< q^{n/2} \sqrt{|A||B|}.
\end{align*}
In particular, $A\cdot B=F$ whenever $\delta(A)\: \delta(B)\geq q^{-(n-2)}$.
\end{cor}
Theorem~\ref{thm: counting field} and Corollary~\ref{cor: general counting field} improve results of Hart-Iosevich \cite{HI} and Hart-Iosevich-Koh-Rudnev \cite[Thm.2.1, Cor.2.4]{HIKR}, respectively S\'ark\"ozy \cite[Thm.1, Cor.1]{Sar} and Gyarmati-S\'ark\"ozy \cite[Thm.2, Cor.2]{GS}. The authors of \cite{HI, HIKR} describe their method as `discrete Fourier analysis', while the authors of \cite{Sar, GS} describe theirs as `character sum estimates'. These two methods have, of course, considerable overlap. Herein the reader will find a proof of Theorem~\ref{thm: counting field} from the perspective of spectral graph theory. A worthwhile point to make is that this is not as distinct a method as it may sound. On the one hand, one often has to deal with character sums when investigating eigenvalues for graphs of algebraic origin. On the other hand, the relevance of eigenvalues to counting problems is a Fourier analytic result. In itself, the idea that Theorem~\ref{thm: counting field} can be approached via spectral graph theory is certainly not new, cf. \cite[Rem.2.2]{HIKR}, Vinh \cite{Vinh} for the problem at hand. The novelty is in the improved bounds and, as we will now explain, in our approach to the spectrum of the relevant graphs.

The counting problem we are interested in suggests that we should consider the following two families of graphs. As usual, $\F^*$ stands for the non-zero elements in $F$; more generally, $(\F^n)^*$ denotes the non-zero vectors in $\F^n$.

\begin{defn}
For $n\geq 2$, let $\UM(F^n)$ denote the bipartite graph on two copies of $(\F^n)^*$, in which vertices $a_\bl$ and $b_\wh$ are adjacent whenever $a\cdot b=1$. For $n\geq 3$, let $\UM_0(F^n)$ denote the bipartite graph on two copies of the projective space $(\F^n)^*/\F^*$, in which vertices $[a]_\bl$ and $[b]_\wh$ are adjacent whenever $a\cdot b=0$. 
\end{defn}

The notation reflects the fact that $\UM(F^n)$ and $\UM_0(F^n)$ are \emph{unimodular graphs}, a name that will be justified later on, when we generalize the construction. The `orthogonality graph' $\UM_0(F^n)$ is, to some extent, a familiar graph. When $n=3$, we recover the point-line incidence graph of the finite projective plane over $\F$. In general, $\UM_0(F^n)$ can be thought of as the point-hyperplane incidence graph of the $(n-1)$-dimensional projective space over $F$. 

The graph $\UM(F^n)$ has half-size $q^n-1$, and it is regular of degree $q^{n-1}$, while $\UM_0(F^n)$ has half-size $(q^n-1)/(q-1)$, and it is regular of degree $(q^{n-1}-1)/(q-1)$. The graphs $\UM(F^n)$ and $\UM_0(F^n)$ are also connected. In fact, both $\UM(F^n)$ and $\UM_0(F^n)$ are Cayley graphs. The diameter of $\UM(F^n)$ is $4$, and the diameter of $\UM_0(F^n)$ is $3$.

For combinatorial applications such as our counting problem, one only needs to know the largest non-trivial eigenvalue of the underlying graphs. In fact, one can compute the entire adjacency spectrum for $\UM(F^n)$ and $\UM_0(F^n)$. We only list the eigenvalues, and we refer to Remark~\ref{rem: uni-multi F} for their multiplicities.

\begin{thm}\label{thm: uni-spectra F}
The eigenvalues of $\UM(F^n)$ are $\pm\: q^{n-1}$, $\pm\: q^{(n-1)/2}$, $\pm\: q^{n/2-1}$.
The eigenvalues of $\UM_0(F^n)$ are $\pm\: (q^{n-1}-1)/(q-1)$, $\pm\: q^{n/2-1}$.
\end{thm}

Observe that both $\UM(F^n)$ and $\UM_0(F^n)$ are pseudo-random in a strong way, namely, they are $d$-regular graphs whose largest non-trivial eigenvalue satisfies $\alpha_2\sim \sqrt{d}$. (In fact, $\alpha_2=\sqrt{d}$ holds for $\UM(F^n)$.) Asymptotically, this is best possible.

One possible approach to Theorem~\ref{thm: uni-spectra F} is to exploit combinatorial features of the two graphs, see \cite[Proof of Thm.2.3]{Alon86} for $\UM_0(F^n)$. Our approach to Theorem~\ref{thm: uni-spectra F} is algebraic, and proceeds as follows. The first step is to use a linear isomorphism between $\F^n$ and $\K$, a field extension of $\F$ of degree $n$, so as to replace the dot product on $\F^n$ by the bilinear form on $\K$ given by the trace of the extension $\K/\F$. The second step is to observe that, in these new realizations of the unimodular graphs over $\F^n$ as trace graphs over $\K$, characters of the multiplicative group $\K^*$ are adjacency eigenvectors. The punch line is that the adjacency eigenvalues turn out to be the signed absolute values of the corresponding \emph{Eisenstein sums}, and these can be computed quite easily. See Sections~\ref{sec: eisenstein, field case} and ~\ref{sec: eigenvalues, explained} for details. Later on, however, we will extend the scope of our unimodular graphs to other finite rings, and it will turn out that the proof sketched above works in far greater generality. It is not clear whether the combinatorial arguments can be adapted, as well. Besides, the algebraic approach has the advantage of providing explicit eigenvectors.

From Theorem~\ref{thm: uni-spectra F} we can easily derive Theorem~\ref{thm: counting field}. Here is an interesting feature of the proof, for which we do not have a good conceptual explanation. The graph that seems most relevant to our counting problem is $\UM(F^n)$. The catch is that $\UM(F^n)$ leads to bounds that are weaker than those claimed in Theorem~\ref{thm: counting field}. We end up using the graph $\UM_0(F^{n+1})$, via a simple embedding trick.

Another application of Theorem~\ref{thm: uni-spectra F} and its proof concerns the isoperimetric constant of the unimodular graphs. There is a well-known lower bound for the isoperimetric constant of a regular graph in terms of the largest non-trivial eigenvalue. The main point of the following result is that, for our unimodular graphs, we can also give upper bounds. These are a by-product of our spectral analysis - specifically, knowledge of eigenvectors plays a crucial role.

\begin{thm}\label{thm: iso UM}
The isoperimetric constant of $\UM(F^n)$ satisfies
\begin{align*} 
&\:\iso \big(\UM(F^n)\big)=\tfrac{1}{2}\big(q^{n-1}-q^{(n-1)/2}\big)\qquad \textrm{ when } n \textrm{ is odd,}\\
\tfrac{1}{2}\big(q^{n-1}-q^{n/2-1}\big)\geq &\:\iso \big(\UM(F^n)\big)\geq\tfrac{1}{2}\big(q^{n-1}-q^{(n-1)/2}\big)\qquad  \textrm{ when } n \textrm{ is even.}
\end{align*}
The isoperimetric constant of $\UM_0(F^n)$ satisfies
\begin{align*}
&\:\iso\big(\UM_0(F^n)\big)\geq \tfrac{1}{2}\Big(\frac{q^{n-1}-1}{q-1}-q^{n/2-1}\Big)\qquad \textrm{ when } n \textrm{ is odd,}\\
&\:\iso\big(\UM_0(F^n)\big)=\tfrac{1}{2}\Big(\frac{q^{n-1}-1}{q-1}-q^{n/2-1}\Big)\qquad \textrm{ when } n \textrm{ is even.}
\end{align*}
\end{thm}
Particularly striking are the cases where we obtain the precise value of the isoperimetric constant. Such exact computations are very rare.

\subsection{Unimodular graphs over finite valuation rings}  Actually, the true goal of this paper is to go well beyond the finite field context, and prove all these results for a certain type of finite rings. A concrete combinatorial motivation is that of obtaining an analogue of Theorem~\ref{thm: counting field} over the ring $\Z/(p^\ell)$.

The unimodular graphs can be defined over any finite ring $R$. Throughout, rings are assumed to be commutative, and to have an identity. Let $R^{n,u}\subseteq R^n$ denote the set of \emph{unimodular $n$-tuples}, namely those tuples whose entries generate $R$ as an ideal. For $n=1$ this is simply the set of units $R^\times$, and we assume $n\geq 2$ in what follows. 

\begin{defn}
For $n\geq 2$, we let $\UM(R^n)$ denote the bipartite graph on two copies of $R^{n,u}$, in which vertices $a_\bl$ and $b_\wh$ are adjacent whenever $a\cdot b=1$. For $n\geq 3$, we let $\UM_0(R^n)$ denote the bipartite graph on two copies of $R^{n,u}/R^\times$, in which vertices $[a]_\bl$ and $[b]_\wh$ are adjacent whenever $a\cdot b=0$. 
\end{defn}

We have to restrict our attention to suitable finite rings if we want to prove substantial facts about the unimodular graphs. Consider the problem of finding their adjacency spectrum. In order to apply the same arguments as those used for fields, we need to consider a class of rings in which we can take appropriate extensions. More importantly, we will need some computations for Eisenstein sums arising from such extensions. This is why we end up focusing on \emph{finite valuation rings}. 

The following are, with some overlap, the main examples of such rings:
\begin{itemize}
\item finite fields,
\item $\Z/(p^\ell)$ where $p\in \Z$ is a prime,
\item $\mathcal{O}/(p^\ell)$ where $\mathcal{O}$ is the ring of integers in a number field and $p\in \mathcal{O}$ is a prime,
\item $F[X]/(f^\ell)$ where $F$ is a finite field and $f\in F[X]$ is an irreducible polynomial.
\end{itemize}

Formally, finite valuation rings are finite rings that are local and principal. The maximal ideal of a finite valuation ring $R$ is of the form $(\pi)$, where the uniformizer $\pi$ is a non-unit of $R$ defined up to a unit of $R$. There are two structural parameters associated to $R$ that play a key role in our results. One is 
\begin{itemize}
\item[$q$ :] the size of the residue field $F=R/(\pi)$, 
\end{itemize}
and the other is
 \begin{itemize}
\item[$\ell$ :] the nilpotency degree of $\pi$, 
\end{itemize}
namely, the smallest positive integer with the property that $\pi^\ell=0$. The lowest possible value, $\ell=1$, occurs precisely when $R$ is a field. For an arbitrary finite valuation ring $R$, the ideal structure is still very simple. It takes the form of a filtration of length $\ell$ (hence the notation):
\begin{align*}
R\supset (\pi) \supset (\pi^2)\supset\dots \supset (\pi^\ell)=0 
\end{align*}
where all the inclusions are strict. Among other things, this filtration implies that $|R|=q^\ell$. In the literature, finite valuation rings are usually called \emph{finite chain rings}---a somewhat less evocative name, in our opinion. 

Until further notice, let $R$ be a finite valuation ring with parameters $q$ and $\ell$ as above. In keeping with our previous convention, $q$ is assumed to be odd. Some properties of the unimodular graphs $\UM(R^n)$ and $\UM_0(R^n)$ are collected in the following statement.

\begin{thm}\label{thm: uni-props}
The following hold.
\begin{itemize}
\item[i)] The bipartite graph $\UM(R^n)$ has half-size $q^{n\ell}-q^{n(\ell-1)}$, and it is regular of degree $q^{(n-1)\ell}$. The bipartite graph $\UM_0(R^n)$ has half-size $q^{(n-1)(\ell-1)}(q^{n}-1)/(q-1)$, and it is regular of degree $q^{(n-2)(\ell-1)}(q^{n-1}-1)/(q-1)$. 
\item[ii)] Both $\UM(R^n)$ and $\UM_0(R^n)$ are Cayley graphs.
\item[iii)] The diameter of $\UM(R^n)$ is $4$, and the diameter of $\UM_0(R^n)$ is $3$. 
\end{itemize}
\end{thm}

We then compute the adjacency eigenvalues of $\UM(R^n)$ and $\UM_0(R^n)$. The proof proceeds as in the case of finite fields. The main work is actually in computing the absolute value of Eisenstein sums arising from unramified extensions of finite valuation rings (Theorem~\ref{thm: abs E}). This is the main technical result of the paper, and it is of independent interest.

\begin{thm}\label{thm: spec UM gen}
The eigenvalues of $\UM(R^n)$ are $\pm\: q^{(n-1)\ell}$, $\pm\: q^{(n-1)\ell-n/2}$, $\pm\: q^{(n-1)(\ell-k/2)}$ for $k=1,\dots,\ell$, as well as $0$ in the case when $R$ is not a field.  The eigenvalues of $\UM_0(R^n)$ are $\pm\: q^{(n-2)(\ell-1)}(q^{n-1}-1)/(q-1)$, $\pm\: q^{(n-2)(\ell-k/2)}$ for $k=1,\dots,\ell$.
\end{thm}

Finally, we apply the spectral insight gained from the previous theorem. The more substantial consequence is the following.

\begin{thm}\label{thm: iso UM gen}
The isoperimetric constant of $\UM(R^n)$ satisfies
\begin{align*} 
&\:\iso \big(\UM(R^n)\big)=\tfrac{1}{2}\:q^{(n-1)\ell}\:\big(1-q^{-(n-1)/2}\big)\qquad \textrm{ when } n \textrm{ is odd,}\\
\tfrac{1}{2}\:q^{(n-1)\ell}\:\big(1-q^{-n/2}\big)\geq &\:\iso \big(\UM(R^n)\big)\geq\tfrac{1}{2}\:q^{(n-1)\ell}\:\big(1-q^{-(n-1)/2}\big)\qquad  \textrm{ when } n \textrm{ is even.}
\end{align*}
The isoperimetric constant of $\UM_0(R^n)$ satisfies
\begin{align*}
&\:\iso\big(\UM_0(R^n)\big)\geq \tfrac{1}{2}\:q^{(n-2)(\ell-1)}\:\Big(\frac{q^{n-1}-1}{q-1}-q^{n/2-1}\Big)\qquad \textrm{ when } n \textrm{ is odd,}\\
&\:\iso\big(\UM_0(R^n)\big)=\tfrac{1}{2}\:q^{(n-2)(\ell-1)}\:\Big(\frac{q^{n-1}-1}{q-1}-q^{n/2-1}\Big)\qquad \textrm{ when } n \textrm{ is even.}
\end{align*}
\end{thm}

The other consequence concerns our starting problem, that of counting solutions to the equation $a\cdot b=r$ for given $r\in R$ and non-empty $A,B\subseteq R^n$.  With the same notations as for finite fields, the following holds.

\begin{thm}\label{thm: counting ring}
We have
\begin{align*}
\big|N_1(A,B) - q^{-\ell}|A||B|\big|< q^{(n-1)(\ell-1/2)} \sqrt{|A||B|}.
\end{align*}
In particular, $1\in A\cdot B$ whenever $\delta(A)\: \delta(B)\geq q^{-(n-1)}$.
\end{thm}

Theorem~\ref{thm: counting ring} significantly generalizes and improves results of Covert-Iosevich-Pakianathan \cite[Thm.1.4]{CIP}, respectively Vinh \cite{Vinh2}. Both \cite{CIP} and \cite{Vinh2} address the particular case $R=\Z/(p^\ell)$. In our density notation, their results are as follows: $1\in  A\cdot A$ whenever $A\subseteq (\Z/(p^\ell))^n$ satisfies $\delta(A)> \ell\: p^{-(n-1)/2}$ (\cite{CIP}); $1\in A\cdot B$ whenever $A,B\subseteq (\Z/(p^\ell))^n$ satisfy $\delta(A)\:\delta(B)\geq 3\:\ell\: p^{-(n-1)}$ (\cite{Vinh2}). By comparison, applying Theorem~\ref{thm: counting ring} to $R=\Z/(p^\ell)$ removes the linear factor $\ell$ in these bounds, thereby making the threshold density independent of $\ell$.

In fact, Theorem~\ref{thm: counting ring} holds for each unit $r\in R^\times$ in place of $1$. For an arbitrary $r\in R$, we may apply Theorem~\ref{thm: counting ring} to $A\times \{1\}, B\times \{1-r\}\subseteq R^{n+1}$, leading to the following statement.

\begin{cor}\label{cor: general counting ring}
For each $r\in R$ we have
\begin{align*}
\big|N_r(A,B) - q^{-\ell}|A||B|\big|< q^{n(\ell-1/2)} \sqrt{|A||B|}.
\end{align*}
In particular, $A\cdot B=R$ whenever $\delta(A)\: \delta(B)\geq q^{-(n-2\ell)}$.
\end{cor}

Note that, in this corollary, the last assertion is empty for $n<2\ell$.

\subsection{Platonic graphs}  The third family of unimodular graphs considered in this paper is that of Platonic graphs. The \emph{Platonic graph} over a finite ring $R$, denoted $\mathrm{Pl}(R)$, has vertex set $R^{2,u}/\{\pm 1\}$, the unimodular pairs taken up to sign. Two vertices $[a,b]$ and $[c,d]$ are adjacent when $ad-bc=\pm 1$. The graph $\mathrm{Pl}(R)$ is a non-bipartite relative of the unimodular graph $\UM(R^2)$. The Platonic graph was first considered by Brooks, Perry, and Petersen~\cite{Bro,BPP} in the case $R=\Z/(p)$. Relevant for \cite{Bro,BPP} is the isoperimetric constant of the Platonic graph. 

Our first result in this direction is the computation of the spectrum in the case when $R$ is a finite valuation ring. In the following theorem, we only list the eigenvalues. Their multiplicities are determined in Remark~\ref{rem: multiplicities platonic R}.

\begin{thm}\label{thm: platonic R}
Let $R$ be a finite valuation ring with parameters $q$ and $\ell$.
\itemize
\item[i)] Assume that $R$ is a field, i.e., $\ell=1$. Then the eigenvalues of $\mathrm{Pl}(R)$ are $q$, $-1$, $\pm\: q^{1/2}$, except for $q=3$ in which case $\pm\: q^{1/2}$ is missing.
\item[ii)] Assume that $R$ is not a field, i.e., $\ell\geq 2$. Then the eigenvalues of $\mathrm{Pl}(R)$ are $q^\ell$, $0$, $\pm\: q^{\ell-k/2}$ for $k=1,\dots,\ell$, except for $q=3$ in which case $\pm\: q^{\ell-1/2}$ is missing.
\end{thm}

Part i) was first proved by Gunnells \cite[Thm.4.2]{Gun} in the original situation when $R=\Z/(p)$, and then by DeDeo, Lanphier, and Minei \cite[Thm.1]{DLM} in general. Their arguments are different, but they both rely on the representation theory of $\mathrm{PSL}_2$ over a finite field. Our approach to Theorem~\ref{thm: platonic R} avoids representation theory. The basic idea is the same as before: we trade $R^2$ for a quadratic extension of $R$. And again, we find that, in the new realization of the Platonic graph, eigenvalues are expressed in terms of Eisenstein sums. In the case of fields, our argument is quite elementary, and much simpler than the representation-theoretic approach of \cite{Gun, DLM}. But our approach also works in a more general context - that of finite valuation rings - where the heavy machinery of representation theory does not. Already the simple case of $R=\Z/(p^\ell)$ is very challenging from a representation-theoretic perspective.

Particularly interesting, however, is the Platonic graph over the ring $\Z/(N)$ where $N$ is an odd positive integer. As explained in \cite{Bro,BPP}, the graph $\mathrm{Pl}(\Z/(N))$ is the $1$-skeleton of a triangulation of the modular curve $X(N)$. In \cite{Gun}, Gunnells observes that the spectrum of $\mathrm{Pl}(\Z/(N))$ contains $-N/p$ for every prime $p$ dividing $N$. Our next goal is to compute the entire spectrum of the Platonic graph over  $\Z/(N)$. In fact, we succeed in doing so over any product $R_1\times\cdots\times R_n$ of finite valuation rings. A first guess would be that the eigenvalues of $\mathrm{Pl}(R_1\times\cdots\times R_n)$ are all the products $\alpha_{(1)}\cdots\alpha_{(n)}$, where each $\alpha_{(i)}$ runs over the eigenvalues of $\mathrm{Pl}(R_i)$ as determined in Theorem~\ref{thm: platonic R}. Unfortunately, the Platonic graph of a product of rings is not the tensor product of the Platonic graphs for the factor rings. It is, however, close enough, and the guess turns out to be partly correct. In general, an explicit list of eigenvalues for $\mathrm{Pl}(R_1\times\cdots\times R_n)$ is somewhat cumbersome to write down. This has to do with the two irregularities revealed by Theorem~\ref{thm: platonic R}. Firstly, $-1$, rather than $0$, is an eigenvalue when $R_i$ is a field. Secondly, there is an `eigenvalue loss' when $R_i$ has $q_i=3$.

So, for the sake of simplicity, we only state the extremal non-trivial eigenvalues of $\mathrm{Pl}(\Z/(N))$. Let us note here that $\mathrm{Pl}(\Z/(N))$ is regular of degree $N$. The case when $N$ is a power of $3$ is already addressed by Theorem~\ref{thm: platonic R}, so we focus on the remaining, generic case.

\begin{thm}\label{thm: spec PLN}
Let $N=3^\ell\: N'$, where $N'>1$ is not divisible by $3$. Let $p'>3$ be the smallest prime dividing $N'$. Then the extremal non-trivial eigenvalues of $\mathrm{Pl}(\Z/(N))$ are as follows:
\begin{align*}
\smallskip
\begin{tabular}{ r  c c }
& \textrm{largest} & \textrm{smallest}\\
\noalign{\smallskip} \hline                       
\noalign{\smallskip}  $\ell=0$\qquad\qquad & $N/\sqrt{p'}$ & $-N/\sqrt{p'}$ \\
 \noalign{\smallskip} $\ell=1$\qquad\qquad  &$N/\sqrt{p'}$ & $-N/\min\{3,\sqrt{p'}\}$ \\
\noalign{\smallskip}  $\ell\geq 2$\qquad\qquad  & $N/\min\{3,\sqrt{p'}\}$ & $-N/\min\{3,\sqrt{p'}\}$ \\
\end{tabular}
\smallskip
\end{align*}
\end{thm}

Recall that, for rather circumstantial reasons, a $d$-regular graph is said to be `Ramanujan' if all its non-trivial adjacency eigenvalues lie in the interval $[-2\sqrt{d-1}, 2\sqrt{d-1}]$. As already noted by Gunnells \cite{Gun}, the graph $\mathrm{Pl}(\Z/(N))$ is usually not Ramanujan for a composite $N$. The same conclusion is reached by Lanphier and Rosenhouse \cite[Thm.3.ii]{LR} with a different approach. Theorem~\ref{thm: spec PLN} and part ii) of Theorem~\ref{thm: platonic R} give a complete answer: the only composite odd numbers $N$ for which $\mathrm{Pl}(\Z/(N))$ is a Ramanujan graph are $N=9,\:15,\:21,\:27,\:33$.

In summary, the spectrum of the Platonic graph $\Z/(N))$ depends, in a somewhat intricate way, on the prime factorization of $N$. At any rate, the following holds.

\begin{cor}
Let $N>1$ be an odd integer and let $p$ be the smallest prime dividing $N$. Then the largest non-trivial eigenvalue of $\mathrm{Pl}(\Z/(N))$ is at most $N/\sqrt{p}$, and so the isoperimetric constant of $\mathrm{Pl}(\Z/(N))$ satisfies
\begin{align*}
\iso\big(\mathrm{Pl}(\Z/(N))\big)\geq \frac{N}{2}\bigg(1-\frac{1}{\sqrt{p}}\bigg).
\end{align*}
\end{cor}

The lower bound we have obtained improves the one obtained in \cite[Thm.3.i]{LR} by combinatorial arguments. We believe that spectral methods can also be used for obtaining upper bounds, but we have not pursued this idea to its very end.

\smallskip
\textbf{Acknowledgements.} I am grateful to Noga Alon for some useful comments.


\section{Eisenstein sums over finite fields}\label{sec: eisenstein, field case}
Let $\F$ be a field with $q$ elements, where $q$ is odd, and let $\K/\F$ be a field extension of degree $n$. The \emph{trace} $\Tr:\K\to \F$ is the map given by $\Tr(s)=\sum_{i=0}^{n-1} s^{q^i}$. What is relevant is not so much the formula, but rather the following property: the trace $\Tr:\K\to \F$ is a surjective $\F$-linear map. Much less important for us, but still mentioned herein, is the multiplicative sibling of the trace. The \emph{norm} $\Nm:\K\to \F$ is given by $\Nm(s)=\prod_{i=0}^{n-1} s^{q^i}$, and it defines a surjective homomorphism $\Nm: \K^*\to \F^*$.

An \emph{Eisenstein sum} for the extension $\K/\F$ is a restricted character sum given by
\begin{align*}
E(\chi)=\sum_{\Tr(s)=1} \chi(s)
\end{align*}
where $\chi$ is a character of the multiplicative group $\K^*$. Equally important in what follows is the `singular' Eisenstein sum
\begin{align*}
E_0(\chi)=\sum_{\substack{\Tr(s)=0 \\ s\neq 0}} \chi(s).
\end{align*}
For the trivial character $\chi_0$ of $\K^*$, we find that $E(\chi_0)=q^{n-1}$ and $E_0(\chi_0)=q^{n-1}-1$. Eisenstein sums defined by non-trivial characters are difficult to compute. Fortunately, all we need to know for the purposes of this paper is their absolute value. 

\begin{thm}\label{thm: field E}
Let $\chi$ be a non-trivial character of $\K^*$.
\begin{itemize}
\item[i)] If $\chi$ is non-trivial on $\F^*$, then:
\begin{align*}
|E(\chi)|=q^{(n-1)/2}, \qquad E_0(\chi)=0.
\end{align*}
\item[ii)] If $\chi$ is trivial on $\F^*$, then:
\begin{align*}
|E(\chi)|=q^{n/2-1}, \qquad |E_0(\chi)|=(q-1)\: q^{n/2-1}.
\end{align*}
Furthermore, $E_0(\chi)=-(q-1)\: E(\chi)$.
\end{itemize}
\end{thm}
This is a known fact, cf. \cite[pp.389--391]{BEW}. A proof can be found in Section~\ref{sec: technical proof}.

\section{Eigenvalues of unimodular graphs over finite fields}\label{sec: eigenvalues, explained}
In this section, we compute the adjacency spectra of $\UM(F^n)$ and $\UM_0(F^n)$ in terms of Eisenstein sums for a degree $n$ extension of $\F$. Recall, the two graphs are defined as follows: $\UM(F^n)$ is the bipartite graph on two copies of $(\F^n)^*$, in which vertices $a_\bl$ and $b_\wh$ are connected if $a\cdot b=1$, while $\UM_0(F^n)$ is the bipartite graph on two copies of the projective space $(\F^n)^*/\F^*$, in which vertices $[a]_\bl$ and $[b]_\wh$ are connected if $a\cdot b=0$. 

These two bipartite graphs, as well as their generalizations to other rings, have the following form. The vertex set consists of two copies, $V_\bl$ and $V_\wh$, of the same set $V$, and two vertices $a_\bl$ and $b_\wh$ are connected when $a$ is suitably related to $b$. The size of $V$ is the \emph{half-size} of the resulting bipartite graph. We work with the \emph{reduced} adjacency operator $A$, mapping complex-valued functions on $V_\wh$ to complex-valued functions on $V_\bl$. The adjacency eigenvalues are then $\pm\: \sqrt{\mu}$ for $\mu$ running over the eigenvalues of $A^*A$, equivalently, of $AA^*$. Both $\sqrt{\mu}$ and $-\sqrt{\mu}$ appear with multiplicity equal to the multiplicity of $\mu$.

\begin{proof}[Proof of Theorem~\ref{thm: uni-spectra F}]
Let $\K$ be an extension of $\F$ of degree $n$. Via an $\F$-linear isomorphism between $\F^n$ and $\K$, we may view the dot product as a non-degenerate $\F$-bilinear form $\beta$ on $\K$. The surjectivity of the trace map implies that there is a unique $\F$-linear isomorphism $\phi: \K\to \K$ such that $\beta(x,y)=\Tr(\phi(x)\:y)$ for all $x,y\in \K$. This allows us to recast the graph $\UM(F^n)$ as the bipartite graph on two copies of $\K^*$, in which vertices $x_\bl$ and $y_\wh$ are connected if $\Tr(\phi(x)\:y)=1$. After a relabeling of, say, the black vertices, we may assume that $\phi$ is the identity map. The resulting bipartite graph on two copies of $\K^*$, in which vertices $x_\bl$ and $y_\wh$ are connected if $\Tr(xy)=1$, is denoted $\Tr(K/F)$. In $\Tr(K/F)$, the neighbours of a vertex $x_\bl$ are $(s/x)_\wh$, where $s\in \K$ runs over all roots of $\Tr(s)=1$. For each character $\chi$ of $\K^*$ we have
\begin{align*}
A\chi(x)=\sum_{\Tr(s)=1} \chi(s/x)=\overline{\chi}(x) \sum_{\Tr(s)=1} \chi(s),
\end{align*}
that is,
\begin{align*}
A\chi=E(\chi)\: \overline{\chi}.
\end{align*}

It follows that the eigenvalues of $\Tr(K/F)$ are $\pm\: |E(\chi)|$, where $\chi$ runs over the characters of $\K^*$. The explicit values are given by Theorem~\ref{thm: field E}.

Similarly, $\UM_0(F^n)$ is isomorphic to the bipartite graph on two copies of $\K^*/\F^*$, in which vertices $[x]_\bl$ and $[y]_\wh$ are connected when $\Tr(xy)=0$. We denote this graph by $\Tr_0(K/F)$. The eigenvalues of $\Tr_0(K/F)$ are $\pm\: |PE_0(\omega)|$, where the `projective' singular Eisenstein sum is given by
\begin{align*}
PE_0(\omega)=\sum_{\substack{\Tr(s)=0\\ s\neq 0}} \omega[s],
\end{align*} 
and $\omega$ runs over the characters of $\K^*/\F^*$. For such an $\omega$, let $\chi_\omega$ be the character of $\K^*$ obtained by composition with the quotient map $K^*\to K^*/F^*$. Then 
\begin{align*}
E_0(\chi_\omega)=\sum_{\substack{\Tr(s)=0\\ s\neq 0}} \chi_\omega(s)=|F^*|\: PE_0(\omega).
\end{align*} 
In conclusion, the eigenvalues of $\Tr_0(K/F)$ are $\pm\: |E_0(\chi)|/(q-1)$, where $\chi$ runs over the characters of $\K^*$ that are trivial on $\F^*$.
\end{proof}

A crucial fact, used implicitly throughout this text, is that characters form a basis for the space of complex-valued functions on an abelian group.

\begin{rem}\label{rem: uni-multi F}
To determine the eigenvalue multiplicities, we do a character count.

For $\Tr(K/F)$, the trivial eigenvalues $\pm\: q^{n-1}$ come from the trivial character, so they each have multiplicity $1$. The eigenvalues $\pm\: q^{n/2-1}$ come from the non-trivial characters of $\K^*$ that are trivial on $\F^*$, and there are exactly $|\K^*|/|\F^*|-1=(q^n-q)/(q-1)$ such characters. The eigenvalues $\pm\: q^{(n-1)/2}$ come from the characters of $\K^*$ that are non-trivial on $\F^*$, and their number is $|\K^*|-|\K^*|/|\F^*|=(q-2)(q^n-1)/(q-1)$.

Consider now the graph $\Tr_0(K/F)$. The trivial eigenvalues $\pm\: (q^{n-1}-1)/(q-1)$ come from the trivial character, so they each have multiplicity $1$; the eigenvalues $\pm\: q^{n/2-1}$ come from the non-trivial characters of $\K^*$ that are trivial on $\F^*$, so they each have multiplicity $(q^n-q)/(q-1)$.
\end{rem}

\begin{rem}
The graph $\Tr(K/F)$, appearing in the previous proof, is a trace analogue of the norm graph introduced by Koll\'ar, R\'onyai, and Szab\'o in \cite{KRS}. In its bipartite form, the norm graph $\mathrm{Nm}(K/F)$ is the bipartite graph on two copies of $\K$, in which vertices $x_\bl$ and $y_\wh$ are connected when $\Nm(x+y)=1$. The graph $\mathrm{Nm}(K/F)$ has half-size $q^n$ and degree $(q^n-1)/(q-1)$. Its eigenvalues are
\begin{align*}
\pm \bigg|\sum_{\Nm(s)=1} \psi(s)\bigg|
\end{align*} 
for $\psi$ running over the additive characters of $\K$ (cf. Alon and Pudl\'ak~\cite{AP}). Unlike the case of Eisenstein sums, it is not known how to compute the above absolute values. Elementary arguments \cite[Ch.II, Thm.3D]{Sch} give an upper bound of $q^{n/2}$ whenever $\psi$ is non-trivial (cf. \cite[Lem.2.3]{AP}).

Projective relatives of norm graphs were considered by Alon, R\'onyai, and Szab\'o in \cite{ARS}. Their eigenvalues turn out to be the signed absolute values of certain Gauss sums \cite{Sz, AR}, and these can be computed explicitly. Although not directly related, the norm graphs - the original ones as well as the projective ones - were a source of inspiration for this paper.
\end{rem}

\section{Algebraic preliminaries on finite valuation rings}
\subsection{Finite local rings, extensions, traces} Let $R$ be a finite local ring with maximal ideal $\pi$. Then $R\setminus \pi=R^\times$, the group of units of $R$. The quotient $F:=R/\pi$ is the \emph{residue field} of $R$. The ring homomorphism $R\to F$ induces a surjective group homomorphism $R^\times \to F^*$, with kernel $1+\pi$.

Certain aspects of the extension theory for finite local rings are of crucial importance to us. We outline the bare minimum, and we refer to \cite[Chapter 4]{BF} for more details. The rough idea is to build extensions of finite local rings by lifting extensions of residue fields. Let $R$, $\pi$, $F$ be as above, and let $K$ be an extension of $F$ of degree $n$. 

\begin{align*}
\xymatrix{
& & K\\
\pi \ar@{->}[r]  &R \ar@{->>}[r] & F\ar@{-}[u]
}\qquad\qquad\qquad
\xymatrix{
\pi S\ar@{->}[r] &S\ar@{->>}[r] & K\\
\pi \ar@{->}[r] \ar@{-}[u] &R \ar@{->>}[r] \ar@{-}[u]& F\ar@{-}[u]
}
\end{align*}

Write $K=F[X]/(\bar{f})$, where $\bar{f}\in F[X]$ is a monic irreducible polynomial of degree $n$. Let $f\in R[X]$ be a monic polynomial of degree $n$ mapping to $\bar{f}$ via $R[X]\to F[X]$. It turns out that $f$ is irreducible (such lifts are said to be \emph{basic irreducible} polynomials). The quotient ring
\begin{align*}
S=R[X]/(f)=R[\xi],
\end{align*}
where $\xi$ denotes the image of $X\in R[X]$ and so $f(\xi)=0$, has the following properties: $S$ is an extension of $R$, $S$ is a finite local ring with maximal ideal $\pi S$, and $S$ has residue field $K$. In what follows, we refer to this extension $R\subseteq S$ as a \emph{standard extension of degree $n$}.

Next, we define the trace map for the extension $R\subseteq S$. This could be done in terms of the Galois group of the extension, but that would require some further theoretical details. A simpler approach is to view $S$ as a free $R$-module of rank $n$, with basis $\{\xi^i: i=0,\dots,n-1\}$. Multiplication by $s\in S$ is an $R$-linear map $S\to S$, and we let $M_s$ be the corresponding $n\times n$ matrix with entries in $R$. The trace $\Tr_{S/R}: S\to R$ is defined by mapping $s\in S$ to the trace of the matrix $M_s$. The same procedure at the level of residue fields recovers the trace map $\Tr_{K/F}:K\to F$ and so the diagram
\begin{align*}
\xymatrix{
S\ar@{->>}[r] \ar@{->}[d]_{\Tr_{S/R}}& K\ar@{->}[d]^{\Tr_{K/F}}\\
R \ar@{->>}[r] & F
}
\end{align*}
is commutative.

\begin{prop}
Let $R\subseteq S$ be a standard extension of finite local rings. Then the trace $\Tr_{S/R}: S\to R$ is $R$-linear, surjective, and it maps $\pi S$ to $\pi$.
\end{prop}

\begin{proof}
$R$-linearity of $\Tr_{S/R}$ is obvious from the definition. Elements of $\pi S$ are finite sums of the form $\sum r_i s_i$, where $r_i\in \pi$ and $s_i\in S$, so $\Tr_{S/R}(\sum r_i s_i)=\sum r_i \:\Tr_{S/R}(s_i)\in \pi$. To prove surjectivity, it suffices to ensure that the image $\Tr_{S/R}(S)\subseteq R$ contains a unit of $R$. Assume that this is not the case. Then $\Tr_{S/R}(S)\subseteq \pi$, so the image of $\Tr_{S/R}(S)$ in the residue field $F$ is $\{0\}$. This means, by the commutativity of the above diagram, that $\Tr_{K/F}$ is identically $0$, a contradiction.
\end{proof}

In what follows, we write $\Tr: S\to R$ for the trace map.

\subsection{Finite valuation rings} Recall from the Introduction that the maximal ideal of a finite valuation ring $R$ is of the form $(\pi)$, where $\pi$ is now an element of $R$, and that we have a filtration by ideals
\begin{align*}
R\supset (\pi) \supset (\pi^2)\supset\cdots \supset (\pi^\ell)=0 
\end{align*}
where $\ell$ denotes the nilpotency degree of $\pi$. For notational reasons, it is sometimes useful to think of $R$ as $(\pi^0)$. There is a natural valuation 
\begin{align*}
\nu: R\to \{0,1,\dots,\ell\}
\end{align*} defined as follows: $\nu(0)=\ell$, and for $r\neq 0$ we set $\nu(r)=k$ if $r\in (\pi^k)\setminus (\pi^{k+1})$. Note that $\nu(r)=k$ if and only if $r=\pi^ku$ for some unit $u\in R^\times$, and that there are precisely $|(\pi^{\ell-k})|$ such representations of $r$.

Each abelian group $(\pi^k)/(\pi^{k+1})$ is a one-dimensional linear space over the residue field $F=R/(\pi)$, so its size is also $q=|F|$. It follows that
\begin{align*}
|(\pi^{k})|=q^{\ell-k}, \qquad k=0,1,\dots,\ell.
\end{align*}
In particular, $|R|=q^\ell$, $|(\pi)|=q^{\ell-1}$, and $|R^\times|=|R|-|(\pi)|=q^\ell-q^{\ell-1}$.

Now let $R\subseteq S$ be a standard extension of degree $n$. Then $S$ is also a finite valuation ring, with uniformizer $\pi$, and its residue field $K$ has size $q^n$. In the following proposition, we establish some useful properties enjoyed by the trace of the extension $R\subseteq S$.

\begin{prop}\label{prop: standard bilinear form}
 Let $R\subseteq S$ be a standard extension of finite valuation rings. Then the trace $\Tr: S\to R$ has the following properties.
\begin{itemize}
\item[i)] $\Tr$ maps $\pi^kS$ onto $\pi^kR$ for each $k=0,\dots,\ell$. 
\item[ii)] If $t\in S$ satisfies $\Tr(ts)=0$ for all $s\in S$, then $t=0$.
\item[iii)] Let $T: S\to R$ be an $R$-linear map. Then there exists a unique $t\in S$ such that $T(s)=\Tr(ts)$ for all $s\in S$.
\item[iv)] Let $\beta$ be a non-degenerate $R$-bilinear form on $S$. Then there exists a unique $R$-linear isomorphism $\phi: S\to S$ such that $\beta(t,s)=\Tr(\phi(t)\:s)$ for all $t,s\in S$.
\end{itemize}
\end{prop}

\begin{proof}
i) We already know that $\Tr(\pi^kS)\subseteq \pi^kR$. Now let $\pi^kr$, where $r\in R$, be an arbitrary element of $\pi^kR$. Surjectivity of $\Tr:S\to R$ provides an $s\in S$ such that $\Tr(s)=r$. Then $\pi^ks\in \pi^kS$ satisfies $\Tr( \pi^k s)= \pi^k \Tr(s)= \pi^kr$.

ii) Assume that $\Tr$ vanishes on $tS$, the ideal generated by $t\in S$. We have $tS=\pi^kS$ for some $k=0,\dots,\ell$, and part i) forces $k=\ell$. Therefore $t=0$.

iii) For each $t\in S$, we have an $R$-linear map $\Tr_t: S\to R$ given by $s\mapsto \Tr(ts)$. By part ii), the assignment $t\mapsto \Tr_t$ defines an injective map from $S$ to the dual $S^*=\mathrm{Hom}_R(S,R)$. But $S$ and $S^*$ have the same size, as $S$ is a free $R$-module, so every element of $S^*$ is of the form $ \Tr_t$.

iv) By the previous part, for each $t\in S$ there is a unique element of $S$, denoted $\phi(t)$, such that $\beta(t,s)=\Tr(\phi(t)\: s)$ for all $s\in S$. It follows that the map $\phi:S\to S$ thus defined is $R$-linear. The non-degeneracy assumption on $\beta$ means that $\phi$ is injective. Thus $\phi: S\to S$ is an $R$-linear isomorphism.
\end{proof}

\section{Eisenstein sums over finite valuation rings} 
Let $R\subseteq S$ be a standard extension of finite valuation rings, of degree $n$, with trace map $\Tr: S\to R$. The Eisenstein sum corresponding to a character $\chi$ of the unit group $S^\times$ is given by
\begin{align*}
E(\chi)=\sum_{\Tr(y)=1} \chi(y).
\end{align*}
Note that an element $y\in S$ satisfying $\Tr(y)=1$ is necessarily a unit. The `singular' Eisenstein sum for $\chi$ is 
\begin{align*}
E_0(\chi)=\sum_{\substack{y\in S^\times \\ \Tr(y)=0}} \chi(y).
\end{align*}
The Eisenstein sums corresponding to the trivial character $\chi_0$ can be easily computed. The value of $E(\chi_0)$ equals the number of solutions for the equation $\Tr(y)=1$, and this number is $|S|/|R|=q^{(n-1)\ell}$.  The value of $E_0(\chi_0)$ equals the number of unit solutions for the equation $\Tr(y)=0$. There are $|S|/|R|=q^{(n-1)\ell}$ solutions in $S$, and $|\pi S|/|\pi R|=q^{(n-1)(\ell-1)}$ solutions in $\pi S$, hence $q^{(n-1)\ell}-q^{(n-1)(\ell-1)}$ solutions which are units.

Our next goal is to compute the absolute values of the Eisenstein sums, that is, to extend Theorem~\ref{thm: field E} from finite fields to finite valuation rings. In order to state our theorem, we must introduce some terminology. We do so at the level of $R$, but we will use it both for $R$ and for $S$.

The filtration $R\supset (\pi) \supset (\pi^2)\supset\cdots \supset (\pi^\ell)=0$ induces a multiplicative filtration for the group of units:
\begin{align*}
R^\times\supset 1+(\pi) \supset 1+(\pi^2)\supset\cdots \supset 1+(\pi^\ell)=1
\end{align*}
In turn, this multiplicative filtration induces a valuation on the characters of $R^\times$. For a multiplicative character $\chi$ we write $\nu(\chi)=k$ when $k$ is smallest with the property that $\chi$ is trivial on $1+(\pi^k)$. By convention, $1+(\pi^0)$ stands for $R^\times$, so $\nu(\chi)=0$ precisely when $\chi$ is trivial. 

\begin{ex}
Let $\e$ be the character of $R^\times$ obtained by lifting the quadratic character of the residue field $F=R/(\pi)$. As $\e$ has order $2$ and the subgroup $1+(\pi)$ has odd order, $\e$ must be trivial on $1+(\pi)$. Thus $\nu(\e)=1$. Note also that $\e$ is the only character of $R^\times$ having order $2$. We call $\e$ \emph{the quadratic character} of $R^\times$.
\end{ex}

\begin{thm}\label{thm: abs E}
Let $\chi$ be a non-trivial character of $S^\times$, with valuation $\nu(\chi)=k\geq 1$. Write $\chi_\res$ for the character of $R^\times$ obtained by restricting $\chi$. 
\begin{itemize}
\item[i)] If $\chi_\res$ is non-trivial, then
\begin{align*}
|E(\chi)|=
\begin{cases}
q^{(n-1)(\ell-k/2)} & \textrm{ if } \nu(\chi_\res)= k\\
0 & \textrm{ if } \nu(\chi_\res)\neq k
\end{cases}, \qquad
E_0(\chi)=0.
\end{align*}
\item[ii)] If $\chi_\res$ is trivial, then
\begin{align*}
|E(\chi)|=
\begin{cases}
q^{(n-1)\ell-n/2} & \textrm{ if } k= 1\\
0 & \textrm{ if } k\neq 1
\end{cases}, \qquad
|E_0(\chi)|=(1-q^{-1})\: q^{(n-1)\ell-(n/2-1)k}.
\end{align*}
Furthermore, if $k=1$ then $E_0(\chi)=-(q-1)\: E(\chi)$.
\end{itemize}
 \end{thm}
 
The proof is deferred to Section~\ref{sec: technical proof}. In the above statement, the valuation of the restricted character, $\nu(\chi_\res)$, is with respect to $R$. We note that $\nu(\chi_\res)\leq \nu(\chi)$. This is the only relation between the two valuations, in the following sense: given $k\geq 1$ and $0\leq j\leq k$, there is a character $\chi$ of $S^\times$ such that $\nu(\chi)=k$ while $\nu(\chi_\res)=j$. This can be shown by a counting argument, similar to the one performed in Remark~\ref{rem: multiplicities platonic R}. 

\begin{ex}\label{ex: quadratic eisen}
Let $\e$ be the quadratic character of $S^\times$. If $n$ is odd then $\e$ restricts to the quadratic character of $R^\times$, so
\begin{align*}
|E(\e)|=q^{(n-1)(\ell-1/2)}, \qquad E_0(\e)=0.
\end{align*}
If $n$ is even then $\e$ restricts to the trivial character of $R^\times$, so
\begin{align*}
|E(\e)|=q^{(n-1)\ell-n/2}, \qquad |E_0(\e)|=(q-1)\:q^{(n-1)\ell-n/2}.
\end{align*}

\end{ex}

\section{Unimodular graphs over finite valuation rings}
Let $R$ be a finite valuation ring. Recall the main parameters: $q$ is the size of the residue field $R/(\pi)$, and $\ell$ is the nilpotency degree of the uniformizer $\pi$. Throughout, we assume that $q$ is odd. This means that $2$ is a unit in $R$.

As $R$ is a local ring, an $n$-tuple $a\in R^n$ is unimodular precisely when some entry of $a$ is a unit in $R$. Recall that, for $n\geq 2$, $\UM(R^n)$ denotes the bipartite graph on two copies of the set of unimodular $n$-tuples of $R$, in which vertices $a_\bl$ and $b_\wh$ are connected whenever $a\cdot b=1$. For $n\geq 3$, $\UM_0(R^n)$ denotes the bipartite graph defined as follows: take two copies of the set of unimodular $n$-tuples of $R$ modulo units of $R$, and join $[a]_\bl$ to $[b]_\wh$ whenever $a\cdot b=0$. 

As in the case of finite fields, the unimodular graphs over $R$ can be thought of as trace graphs associated to extensions of $R$.

\begin{lem} Let $S$ be a standard extension of $R$ of degree $n$, with trace map $\Tr: S\to R$.
\begin{itemize}
\item[i)] Let $\Tr(S/R)$ denote the bipartite graph on two copies of $S^\times$, in which vertices $x_\bl$ and $y_\wh$ are connected if $\Tr(xy)=1$. Then $\Tr(S/R)$ is isomorphic to $\UM(R^n)$.
\item[ii)] Let $\Tr_0(S/R)$ denote the bipartite graph on two copies of $S^\times/R^\times$, in which vertices $[x]_\bl$ and $[y]_\wh$ are connected if $\Tr(xy)=0$. Then $\Tr_0(S/R)$ is isomorphic to $\UM_0(R^n)$.
\end{itemize}
\end{lem}
\begin{proof}
i) The map $(a_1,a_2,\dots,a_n)\mapsto a_1+a_2\xi+\cdots+a_n\xi^{n-1}$ is an $R$-linear isomorphism between the free $R$-modules $R^n$ and $S=R[\xi]$. Under this isomorphism, unimodular $n$-tuples correspond to units of $S$, and the dot product on $R^n$ turns into a non-degenerate $R$-bilinear form $\beta$ on $S$. By part iv) of Proposition~\ref{prop: standard bilinear form}, we have $\beta(t,s)=\Tr(\phi(t)\: s)$ for some $R$-linear isomorphism $\phi:S\to S$. By $R$-linearity, $\phi$ maps $\pi S$ to $\pi S$, and it does so in a bijective fashion. It follows that $\phi$ restricts to a permutation of the units of $S$. After a relabeling of, say, the black vertices, we obtain the graph $\Tr(S/R)$, which is therefore an isomorphic copy of $\UM(R^n)$.

ii) The $R$-linear nature of the above arguments allows us to mod out by the units of $R$, yielding an isomorphism between $\UM_0(R^n)$ and $\Tr_0(S/R)$.
\end{proof}

Both pictures of the unimodular graphs, the `standard' picture and the `tracial' picture, have their own advantages. We will eventually compute the eigenvalues of the unimodular graphs by viewing them as trace graphs. On the other hand, there are natural full embeddings  
\begin{align*}
\UM(R^n)\into \UM(R^{n+1})& \qquad a_{\bl/\wh}\mapsto (a,0)_{\bl/\wh}\\
\UM_0(R^n)\into \UM_0(R^{n+1}) & \qquad [a]_{\bl/\wh}\mapsto [a,0]_{\bl/\wh}\\
\UM(R^n)\into \UM_0(R^{n+1}) & \qquad a_{\bl}\mapsto [a,1]_\bl,\: b_{\wh}\mapsto [b,-1]_\wh
\end{align*}
which are obvious in the standard picture, but obscure in the tracial picture. The following proof presents further evidence that having both pictures at hand is very useful.

\begin{proof}[Proof of Theorem~\ref{thm: uni-props}]
i) Either picture can be used for this fairly straightforward counting. Let us do it in terms of the trace graphs. The graph $\Tr(S/R)$ has half-size $|S^\times|=(q^{n})^\ell-(q^{n})^{(\ell-1)}$. The degree of each vertex equals the number of solutions for the equation $\Tr(y)=1$, and we have already seen that this number is $q^{(n-1)\ell}$. The graph $\Tr_0(S/R)$ has half-size
\begin{align*}
\frac{|S^\times|}{|R^\times|}=\frac{q^{n\ell}-q^{n(\ell-1)}}{q^{\ell}-q^{\ell-1}}=q^{(n-1)(\ell-1)}\:\frac{q^{n}-1}{q-1}.
\end{align*} 
The degree of each vertex equals the number of solutions in $S^\times/R^\times$ for the equation $\Tr([y])=0$. We have already counted the number of solutions in $S^\times$ as $q^{(n-1)\ell}-q^{(n-1)(\ell-1)}$. Hence the degree of each vertex equals
 \begin{align*}
\frac{q^{(n-1)\ell}-q^{(n-1)(\ell-1)}}{q^{\ell}-q^{\ell-1}}=q^{(n-2)(\ell-1)}\:\frac{q^{n-1}-1}{q-1}.
\end{align*} 

ii) We show that the trace graphs $\Tr(S/R)$ and $\Tr_0(S/R)$ are Cayley graphs. Inversion in the group $S^\times$ defines a semidirect product $S^\times \rtimes \{\pm 1\}$. Concretely, the multiplication is $(x,\sigma)(y,\tau)=(xy^{\sigma},\sigma\tau)$ for $x,y\in S^\times$ and $\sigma,\tau\in \{\pm 1\}$. Now consider the following subset of $S^\times \rtimes \{\pm 1\}$:
\begin{align*}
X=\{(g,-1): g\in S^\times, \Tr(g)=1\}
\end{align*} 
Note that $X$ does not contain the neutral element $(1,1)$, and $X$ is a symmetric subset as it consists of elements of order $2$. In the Cayley graph of $S^\times \rtimes \{\pm 1\}$ with respect to $X$, each edge connects a vertex in $S^\times\times \{+1\}$ to a vertex in $S^\times\times \{-1\}$. More precisely, $(x,+1)$ is connected to $(y,-1)$ if and only if $(y,-1)=(x,+1)(g,-1)=(xg,-1)$ for some $g$ satisfying $\Tr(g)=1$, i.e., $\Tr(y/x)=1$. In other words, the Cayley graph of $S^\times \rtimes \{\pm 1\}$ with respect to $X$ is the bipartite graph on two copies of $S^\times$, in which vertices $x_\bl$ and $y_\wh$ are connected if $\Tr(y/x)=1$. Up to a relabeling of the black vertices by inversion, we have recovered $\Tr(S/R)$.

An adaptation of the previous argument shows that $\Tr_0(S/R)$ is the Cayley graph of the semidirect product $(S^\times/R^\times) \rtimes \{\pm 1\}$ with respect to the subset
\begin{align*}
X_0=\{([g],-1): [g]\in S^\times/R^\times, \Tr(g)=0\}.
\end{align*} 
The next part of the proof implies that $\Tr(S/R)$ and $\Tr_0(S/R)$ are connected, so $X$ and $X_0$ are in fact generating subsets for the corresponding groups.

iii) Here we view $\UM(R^n)$ and $\UM_0(R^n)$ in their original form.

We start with $\UM(R^n)$. By vertex-transitivity, it suffices to find the distance between the vertex $(1,0,\dots,0)_\bl$ and an arbitrary vertex $b_\bl$ or $b_\wh$, where $b=(b_1,b_2,\dots,b_n)$ is unimodular.

\emph{Case 1:} one of $b_2,\dots,b_n$ is a unit. Then the distance to $b_\bl$ is at most $2$, while the distance to $b_\wh$ is at most $3$. Indeed, we may assume without loss of generality that $b_2$ is a unit. We then have the following paths in $\UM(R^n)$:
\begin{align*}
&(1,0,\dots,0)_\bl \sim (1,(1-b_1)/b_2,0,\dots,0)_\wh \sim (b_1, b_2,\dots,b_n)_\bl\\
&(1,0,\dots,0)_\bl \sim (1,b_2,0,\dots,0)_\wh \sim (0,1/b_2,0,\dots,0)_\bl \sim (b_1, b_2,b_3,\dots,b_n)_\wh
\end{align*}

\emph{Case 2:} $b_1$ is a unit. Let $b'=((1-b_2)/b_1,1,0,\dots,0)$, a unimodular tuple to which the previous case applies. As $b_\wh\sim b'_\bl$ and $b_\bl\sim b'_\wh$, we infer that the distance to $b_\wh$ is at most $3$ while the distance to $b_\bl$ is at most $4$. Let us analyze the possibility that $b_\bl$ is at most two edges away:
\begin{align*}
&(1,0,\dots,0)_\bl \sim (1,x_2,\dots,x_n)_\wh \sim (b_1, b_2,\dots,b_n)_\bl
\end{align*}
where $b_1+b_2x_2+\cdots+b_nx_n=1$. But this cannot hold if $b_2,\dots, b_n\in(\pi)$ and $b_1$ is a unit that is not in the subgroup $1+(\pi)$. We conclude that $\UM(R^n)$ has diameter $4$.

Next, we turn to $\UM_0(R^n)$, and we argue along similar lines. By vertex-transitivity, it suffices to find the distance between the vertex $[1,0,\dots,0]_\bl$ and an arbitrary vertex $[b]_\bl$ or $[b]_\wh$, where $b=(b_1,b_2,\dots,b_n)$ is unimodular.

\emph{Case 1:} one of $b_2,\dots,b_n$ is a unit. Then the distance to $[b]_\bl$ is at most $2$, while the distance to $[b]_\wh$ is at most $3$. Indeed, if, say, $b_2$ is a unit, then have the following paths in $\UM_0(R^n)$:
\begin{align*}
&[1,0,\dots,0]_\bl \sim [0,-b_3,b_2,0,\dots,0]_\wh \sim [b_1, b_2,b_3,\dots,b_n]_\bl\\
&[1,0,\dots,0]_\bl \sim [0,b_2,b_3,0,\dots,0]_\wh \sim[0,-b_3,b_2,0,\dots,0]_\bl \sim [b_1, b_2,b_3,\dots,b_n]_\wh
\end{align*}

\emph{Case 2:} $b_1$ is a unit. Let $b'=(-b_2, b_1,0,0,\dots,0)$, a unimodular tuple to which the previous case applies. As $[b]_\wh\sim [b']_\bl$ and $[b]_\bl\sim [b']_\wh$, the distance to $[b]_\wh$ is at most $3$ while the distance to $[b]_\bl$ is at most $4$. It remains to analyze  the possibility that the distance to $[b]_\bl$ is at most $2$:
\begin{align*}
&[1,0,\dots,0]_\bl \sim [0,x_2,\dots,x_n]_\wh \sim [b_1, b_2,\dots,b_n]_\bl
\end{align*}
where $b_2x_2+\cdots+b_nx_n=0$, and one of $x_2,\dots,x_n$ is a unit. This is equivalent to one of $b_2,\dots,b_n$ belonging to the ideal generated by the others. As $R$ is principal, this is indeed the case, and we conclude that $\UM_0(R^n)$ has diameter $3$.
\end{proof}

\begin{rem}
The girth can also be determined. We only indicate the results, leaving the details to the reader. The girth of $\UM_0(R^n)$ is $4$, except when $n=3$ and $R$ is a field in which case the girth is $6$. Likewise, the girth of $\UM(R^n)$ is $4$, except when $n=2$ and $R$ is a field in which case the girth is $6$.
\end{rem}

\begin{rem}
Over a finite local ring $R$, it is still the case that $\UM(R^n)$ has diameter $4$. The diameter of $\UM_0(R^n)$, on the other hand, reveals a small surprise. Let $n_R$ denote the smallest positive integer with the property that every ideal of $R$ can be generated by $n_R$ elements. Thus, finite valuation rings are characterized by $n_R=1$. Then the diameter of $\UM_0(R^n)$ is $3$ for $n\geq n_R+2$, and $4$ otherwise. This can be glimpsed from the last step in the proof of the previous theorem.
\end{rem}

\begin{proof}[Proof of Theorem~\ref{thm: spec UM gen}]
The argument is the same as in the case of finite fields, see the proof of Theorem~\ref{thm: uni-spectra F}. The eigenvalues of the trace graph $\Tr(S/R)$ are $\pm |E(\chi)|$, where $\chi$ runs over the characters of $S^\times$. The signed absolute values of the Eisenstein sums are given by Theorem~\ref{thm: abs E}.

The eigenvalues of the trace graph $\Tr_0(S/R)$ are $\pm |E_0(\chi)|/|R^\times|$, where $\chi$ runs over the characters of $S^\times$ that are trivial on $R^\times$. Theorem~\ref{thm: abs E} provides the signed absolute values of the singular Eisenstein sums. When $\chi$ is non-trivial, we compute
\begin{align*}
\frac{|E_0(\chi)|}{|R^\times|}=\frac{(q-1)\:q^{(n-1)\ell-(n/2-1)k-1}}{(q-1)\:q^{\ell-1}}=q^{(n-2)(\ell-k/2)}
\end{align*} 
for $k=1,\dots,\ell$.
\end{proof}

Again, the eigenvalue multiplicities can be determined by doing a character count, but the formulas are not particularly appealing, and we will not need these multiplicities. The method is illustrated in Remark~\ref{rem: multiplicities platonic R} for the case of Platonic graphs.

\section{Applications}
\subsection{Edge counting} Let us recall a simple, but powerful estimate for edge-counting, due to Alon and Chung \cite{AC}.

Let $X$ be a connected, $d$-regular and bipartite graph on two copies of $V$, where $|V|=m$. Given two non-empty vertex subsets $U\subseteq V_\bl$ and $W\subseteq V_\wh$, we let $e(U,W)$ denote the number of edges joining vertices in $U$ to vertices in $W$. Then:
\begin{align}\label{eq: alon-chung}
\Big|e(U,W)-\frac{d}{m} |U||W|\Big|\leq \frac{\alpha_2}{m} \sqrt{|U||W|(m-|U|)(m-|W|)} 
\end{align}
where $\alpha_2$ is the largest non-trivial adjacency eigenvalue of $X$.

\begin{proof}[Proof of Theorem~\ref{thm: counting ring}] Still keeping the previous notations, let us first put \eqref{eq: alon-chung} in a form that is more convenient for our present needs. The estimate $\sqrt{(m-|U|)(m-|W|)}\leq m-\sqrt{|U||W|}$ leads to
\begin{align*}
 \frac{d+\alpha_2}{m}|U||W|-\alpha_2\sqrt{|U||W|}\leq e(U,W)\leq  \frac{d-\alpha_2}{m}|U||W|+\alpha_2\sqrt{|U||W|}.
\end{align*}
Therefore
\begin{align}\label{eq: alon-chung 2}
\big|e(U,W)- c\: |U||W|\big|< \alpha_2\sqrt{|U||W|} \qquad \textrm{ whenever }\frac{d-\alpha_2}{m}< c<  \frac{d+\alpha_2}{m}.
\end{align}
The graph $\UM_0(R^n)$ has 
\begin{align*}
m=q^{(n-1)(\ell-1)}\:\frac{q^n-1}{q-1},\quad d=q^{(n-2)(\ell-1)}\:\frac{q^{n-1}-1}{q-1},\quad \alpha_2=q^{(n-2)(\ell-1/2)}
\end{align*} 
and a small computation shows that
\begin{align*}
\frac{d\pm \alpha_2}{m}=q^{-\ell}\: \frac{\frac{q^n-1}{q-1}-1\pm q^{n/2}}{\frac{q^n-1}{q-1}}.
\end{align*}
We may thus use $c=q^{-\ell}$ in \eqref{eq: alon-chung 2}.

Now take $A,B\subseteq R^n$. Let $A'$ and $B'$ be the projective images of $A\times \{1\}$ and $B\times \{-1\}$ in $R^{n+1,u}/R^\times$. Indeed, $A\times \{1\}$ and $B\times \{-1\}$ consist of unimodular tuples, thanks to the last coordinate. Note also that $|A'|=|A|$ and $|B'|=|B|$. Viewing $A'$ and $B'$ as black, respectively white vertices in the graph $\UM_0(R^{n+1})$, the number of edges between $A'$ and $B'$ is precisely $N_1(A,B)$. The desired estimate follows from \eqref{eq: alon-chung 2}.
\end{proof}

\subsection{Isoperimetric constant} Recall that the isoperimetric constant of a graph $X$, herein assumed to be connected and $d$-regular, is given by
\begin{align*}
\iso(X)=\min \frac{e(U,W)}{\min\{|U|, |W|\}}
\end{align*}
where the minimum is taken over all partitions of the vertices of $X$ into two non-empty sets $U$ and $W$. As before, $e(U,W)$ denotes the number of edges connecting vertices in $U$ to vertices in $W$.

The isoperimetric constant can be estimated with the help of eigenvalues. Firstly, there is a well-known lower bound, usually attributed to Alon and Milman, saying that
\begin{align}\label{eq: lower bound}
\iso(X)\geq \tfrac{1}{2}(d-\alpha_2)
\end{align}
where $\alpha_2$ is the largest non-trivial eigenvalue of $X$. Secondly, a seminal idea due to Fiedler and Donath-Hoffman, from the early seventies, is that one can partition the vertices of $X$ by using an adjacency eigenvector. In favorable circumstances, this leads to an upper bound for $\iso(X)$ in terms of the corresponding eigenvalue. 

We will use the latter idea in the following form.

\begin{lem}
Let $X$ be a connected, $d$-regular, and bipartite graph on two copies of $V$, where $|V|=m$. Assume that the reduced adjacency matrix has an eigenvector $f$ which is $\{\pm 1\}$-valued and has zero mean, $\sum_{v\in V} f(v)=0$, with eigenvalue $\alpha$. (A fortiori, $\alpha$ is integral and $|\alpha|<d$.) Then
\begin{align}\label{eq: upper bound}
\iso(X)\leq \tfrac{1}{2}(d-|\alpha|).
\end{align}
\end{lem}

\begin{proof}
We aim for a partition of $V_\wh\cup V_\bl$ into two sets $U$ and $W$ such that $|U|=|W|=m$, and $e(U,W)=\tfrac{1}{2}m(d-|\alpha|)$. The eigenvector $f$ partitions $V$ into two subsets, $V(\sigma)=\{v: f(v)=\sigma\}$ for $\sigma\in \{\pm 1\}$. The characteristic function of $V(\sigma)$ is $\tfrac{1}{2}(f_1+\sigma f)$, where $f_1$ denotes the constant function equal to $1$ on $V$. Note that $f_1$ is orthogonal to $f$.

Let $A$ denote the reduced adjacency matrix. For $\sigma_1,\sigma_2\in \{\pm 1\}$, the number of edges between $V(\sigma_1)_\wh$ and $V(\sigma_2)_\bl$ is
\begin{align*}
\big\la A\: \tfrac{1}{2}(f_1+\sigma_1 f), \tfrac{1}{2}(f_1+\sigma_2 f) \big\ra&= \tfrac{1}{4} \la df_1+\sigma_1 \alpha f, f_1+\sigma_2 f \ra\\
&=  \tfrac{1}{4} \big( d\:\la f_1,f_1\ra+\sigma_1 \sigma_2 \alpha\: \la f, f \ra \big)= \tfrac{1}{4} m ( d+ \sigma_1 \sigma_2 \alpha).
\end{align*}
Now pick $\sigma\in \{\pm 1\}$ to be the sign of $\alpha$, so $\sigma \alpha=|\alpha|$. The desired partition is given by $U=V(+1)_\wh\cup V(\sigma)_\bl$ and $W=V(-1)_\wh\cup V(-\sigma)_\bl$. 
\end{proof}

\begin{proof}[Proof of Theorem~\ref{thm: iso UM gen}]
The lower bounds for the isoperimetric constant come from \eqref{eq: lower bound}. Focusing on the upper bounds, we wish to apply  \eqref{eq: upper bound}. Consider the graphs $\UM(R^n)$ and $\UM_0(R^n)$ in their trace realizations, $\Tr(S/R)$ and $\Tr_0(S/R)$. The role of $f$ in the previous lemma is played by $\e$, the quadratic character of $S^\times$. In $\Tr(S/R)$, $\e$ is a reduced adjacency eigenvector with eigenvalue $E(\e)$, and
\begin{align*}
|E(\e)|=
\begin{cases}
q^{(n-1)(\ell-1/2)} & \textrm{ if } n \textrm{ is odd},\\
q^{(n-1)\ell-n/2} & \textrm{ if } n \textrm{ is even}.
\end{cases}
\end{align*}
as explained in Example~\ref{ex: quadratic eisen}.

If $n$ is even, then $\e$ is also a reduced adjacency eigenvector for $\Tr_0(S/R)$, with eigenvalue $E_0(\e)/|R^\times|$. We read off the absolute value of $E_0(\e)$ from Example~\ref{ex: quadratic eisen}, and we obtain
\begin{align*}
|E_0(\e)|/|R^\times|=q^{(n-2)(\ell-1)+n/2-1}.
\end{align*}
\end{proof}

\section{Platonic graphs}
Let $R$ be a finite ring. Recall that a pair $(a,b)\in R^2$ is unimodular if the ideal generated by $a$ and $b$ is the whole of $R$. To phrase this in a way that is consistent with the following discussion, $(a,b)\in R^2$ is unimodular if there are $c,d\in R$ such that $ad-bc=1$. The Platonic graph $\mathrm{Pl}(R)$ has vertex set $R^{2,u}/\{\pm 1\}$, and two vertices $[a,b]$ and $[c,d]$ are connected whenever $ad-bc=\pm 1$. 

We consider an operator which, on the one hand, is closely related to the adjacency operator of $\mathrm{Pl}(R)$, and, on the other hand, is well-behaved under ring products. Let
\begin{align*}
D: \mathcal{F}(R^{2,u})\to \mathcal{F}(R^{2,u}), \qquad Df(a,b)=\sum_{(c,d):\: ad-bc=1} f(c,d)
\end{align*}
where $\mathcal{F}(R^{2,u})$ denotes the linear space of complex-valued functions on $R^{2,u}$. Although we will not rely on this perspective, we note here that $D$ is the reduced adjacency operator of the following graph, isomorphic to the unimodular graph $\UM(R^2)$: the vertex set consists of two copies of $R^{2,u}$, and $(a,b)_\bl$ is connected to $(c,d)_\wh$ if $ad-bc=1$.

A function $f\in\mathcal{F}(R^{2,u})$ is \emph{even} if $f(-a,-b)=f(a,b)$ for all $(a,b)\in R^{2,u}$, respectively \emph{odd} if $f(-a,-b)=-f(a,b)$ for all $(a,b)\in R^{2,u}$. The corresponding subspaces of $\mathcal{F}(R^{2,u})$ are denoted $\mathcal{F}(R^{2,u})^+$ and $\mathcal{F}(R^{2,u})^-$. Then $D$ respects the decomposition $\mathcal{F}(R^{2,u})=\mathcal{F}(R^{2,u})^+\oplus\mathcal{F}(R^{2,u})^-$,
so $D$ decomposes as $D=D^+\oplus D^-$. Under the identification $\mathcal{F}(R^{2,u})^+=\mathcal{F}(R^{2,u}/\{\pm 1\})$, the even part $D^+$ is precisely the adjacency operator of $\mathrm{Pl}(R)$.

\begin{lem}\label{lem: D+D-}
Let $R$ be a finite valuation ring.
\begin{itemize}
\item[i)] Assume $R$ is a field, i.e., $\ell=1$. Then $D^+$ has eigenvalues $q$, $-1$, $\pm\: q^{1/2}$, except when $q=3$ in which case $\pm\: q^{1/2}$ is missing, while $D^-$ has eigenvalues $\pm\: \mathrm{i}q^{1/2}$.
\item[ii)] Assume $R$ is not a field, i.e., $\ell\geq 2$. Then $D^+$ has eigenvalues $q^\ell$, $0$, $\pm\: q^{\ell-k/2}$ for $k=1,\dots,\ell$, except when $q=3$ in which case $\pm\: q^{\ell-1/2}$ is missing, while $D^-$ has eigenvalues $0$,$\pm\: \mathrm{i} q^{\ell-k/2}$ for $k=1,\dots,\ell$.
\end{itemize}
\end{lem}

\begin{proof}
Let $S=\{a+b\sqrt{j}: a, b\in R\}$ be a standard quadratic extension of $R$. Concretely, $j\in R^\times$ is a lift of a non-square in the residue field $F$ of $R$. The trace $\Tr: S\to R$ is given by $\Tr(s)=s+\overline{s}$, where conjugation in $S$ means exactly what it should: $\overline{a+b\sqrt{j}}=a-b\sqrt{j}$. There is also a norm map, given by $\Nm: S\to R$ and $\Nm(s)=s\overline{s}$. As $N$ is multiplicative, it restricts to a group homomorphism $N:S^\times\to R^\times$. Let us argue that, as in the case of finite fields, this homomorphism is onto. Let $r\in R^\times$. We have to show that $a^2-jb^2=r$ for some $a,b\in R$. At the level of residue fields, the norm is surjective, so there are $a_0,b_0\in R$ such that $a_0^2-jb_0^2=r$ mod $\pi$. At least one of $a_0$ and $b_0$, say $b_0$, is a unit in $R$. Then $b_0$ is a simple root for the equation $a_0^2-jx^2=r$ mod $\pi$ so, by Hensel's lemma, there exists $b\in R$, $b=b_0$ mod $\pi$, such that $a^2-jb^2=r$.

Via the correspondence $(a,b)\leftrightarrow a+b\sqrt{j}$, unimodular pairs for $R$ correspond to units in $S$. So we may think of $D$ as the operator $D:\mathcal{F}(S^\times)\to \mathcal{F}(S^\times)$ given by
\begin{align*}
Df(x)=\sum_{y: \: \Tr((2\sqrt{j})^{-1}\overline{x}y)=1} f(y)
\end{align*}
On a character $\chi$ of $S^\times$, the operator $D$ acts as follows:
\begin{align*}
D\chi(x)=\sum_{y: \: \Tr((2\sqrt{j})^{-1}\overline{x}y)=1} \chi(y)= \sum_{\Tr(y)=1} \chi\Big(\frac{2\sqrt{j}\: y}{\overline{x}}\Big)= \chi(2\sqrt{j})\:E(\chi)\:\overline{\chi}(\overline{x})
\end{align*}
If $\chi$ is a character of $S^\times$, then $\chi^*$ defined by $x\mapsto \overline{\chi}(\overline{x})$ is again a character of $S^\times$. With this notation, we have 
\begin{align*}
D\chi= c(\chi)\: \chi^*, \qquad c(\chi):=\chi(2\sqrt{j})\: E(\chi).
\end{align*} 
A character $\chi$ is even if $\chi(-1)=\chi(1)$, respectively odd if $\chi(-1)=-\chi(1)$. Note that $\chi\mapsto \chi^*$ is parity-preserving, and that $ \chi^{**}=\chi$. This means that $D$, viewed as a matrix, is block-diagonal. One block is a diagonal matrix indexed by the characters which are fixed under the transformation $\chi\mapsto \chi^*$. The remaining blocks are $2$-by-$2$ off-diagonal matrices, one for each pair $\{\chi,\chi^*\}$ of non-fixed characters.

We have $\chi^*=\chi$ if and only if $\chi(x\overline{x})=\chi(\Nm(x))=1$ for all $x\in S^\times$, i.e., $\chi$ is trivial on $R^\times$. In particular, $\chi$ is even. For the trivial character we have $c(\chi_0)=E(\chi_0)=q^\ell$. Now let $\chi$ be non-trivial, but trivial on $R^\times$. Recall from Theorem~\ref{thm: abs E} that $E(\chi)=0$ if $\nu(\chi)\neq 1$ and $E_0(\chi)=-(q-1)\: E(\chi)$ if $\nu(\chi)= 1$. But the singular Eisenstein sum turns out to be easily computable:
\begin{align*}
E_0(\chi)=\sum_{\substack{\Tr (y)=0\\ y\in S^\times}} \chi(y)= \sum_{b\in R^\times} \chi(b\sqrt{j})=|R^\times|\:\chi(\sqrt{j})= (q^\ell-q^{\ell-1})\:\chi(\sqrt{j})
\end{align*}
Thus, if $\nu(\chi)=1$, then
\begin{align*}
c(\chi)=E(\chi)\:\chi(2\sqrt{j})=-q^{\ell-1}\:\chi(\sqrt{j})\:\chi(2\sqrt{j})=-q^{\ell-1}\:\chi(2j)=-q^{\ell-1}.
\end{align*}
Now let $\chi$ be a character which is not fixed under $\chi\mapsto \chi^*$, i.e., $\chi$ is non-trivial on $R^\times$. As $E(\chi^*)=E(\overline{\chi})$ and $\chi^*(2\sqrt{j})=\overline{\chi}(-2\sqrt{j})=\pm\:\overline{\chi}(2\sqrt{j})$, according to the parity of $\chi$, we have 
\begin{align*}
c(\chi^*)=\pm\:\overline{c(\chi)}.
\end{align*} 
The $2$-by-$2$ block corresponding to the pair $\{\chi,\chi^*\}$ has the form
\begin{align*}
\begin{pmatrix}
0 & \pm\:\overline{c(\chi)}\\
c(\chi) & 0
\end{pmatrix}
\end{align*}
and the resulting eigenvalues are $\pm\: |c(\chi)|=\pm\: |E(\chi)|$ when $\chi$ is even, respectively  $\pm\: \mathrm{i} |c(\chi)|=\pm\: \mathrm{i} |E(\chi)|$ when $\chi$ is odd. By Theorem~\ref{thm: abs E}, we know that $E(\chi)=0$ when $\nu(\chi)\neq \nu(\chi_\res)$, and $|E(\chi)|=q^{\ell-k/2}$ when $\nu(\chi)=\nu(\chi_\res)=k$.
\end{proof}

\begin{rem}\label{rem: multiplicities platonic R}
Let us count the multiplicities.

i) Let $\ell=1$, i.e., $R$ is a field.

\emph{Multiplicities for $D^+$}. There are $\tfrac{1}{2}(q^2-1)$ even characters of $S^\times$. Among them, we have $|S^\times|/|R^\times|=q+1$ characters that are trivial on $R^\times$. The trivial character yields the eigenvalue $q$, and the remaining $q$ characters are eigenvectors for the eigenvalue $-1$. There are $\tfrac{1}{2}\big(\tfrac{1}{2}(q^2-1)-(q+1)\big)=\tfrac{1}{4}(q+1)(q-3)$ pairs of characters that are not trivial on $R^\times$, and this is the multiplicity for both $q^{1/2}$ and $-q^{1/2}$. When $q=3$, this multiplicity is $0$.

\emph{Multiplicities for $D^-$}. There are $\tfrac{1}{2}(q^2-1)$ odd characters of $S^\times$. Both $\mathrm{i} q^{1/2}$ and $-\mathrm{i} q^{1/2}$ have multiplicity $\tfrac{1}{4}(q^2-1)$.

ii) Let $\ell\geq 2$, i.e., $R$ is not a field.

\emph{Multiplicities for $D^+$}. It suffices to focus on the non-zero eigenvalues since the multiplicity of $0$ can be determined from the dimension count. The trivial eigenvalue $q^\ell$ comes from the trivial character, and it has multiplicity $1$. 

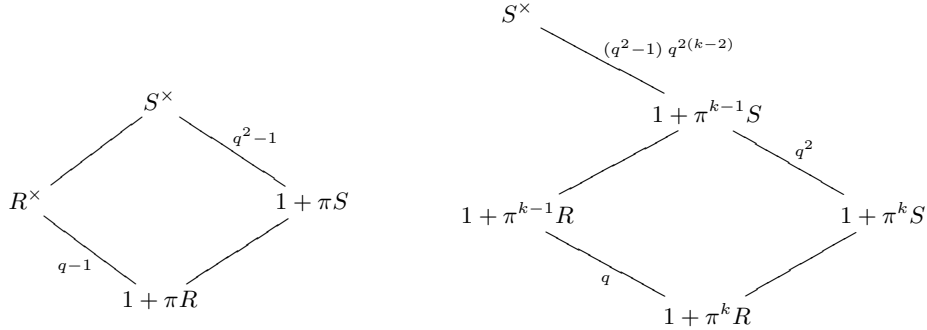
\begin{align*}
\xymatrix{
&&\\
& S^\times&\\
R^\times\ar@{-}[ru] && 1+\pi S\ar@{-}[lu]_{q^2-1}\\
& 1+\pi R\ar@{-}[ru] \ar@{-}[lu]^{q-1} &
}\qquad\qquad 
\xymatrix{
S^\times &&\\
& 1+\pi^{k-1}S \ar@{-}[lu]_{\quad(q^2-1)\: q^{2(k-2)}} &\\
1+\pi^{k-1}R \ar@{-}[ru] && 1+\pi^{k}S\ar@{-}[lu]_{q^2}\\
& 1+\pi^k R\ar@{-}[ru] \ar@{-}[lu]^{q} &
}
\end{align*}

The even characters of $S^\times$ that are non-trivial on $R^\times$ yield eigenvalues $\pm\: q^{\ell-k/2}$ when $\nu(\chi_\res)=\nu(\chi)=k$. For $k=1$, the condition $\nu(\chi_\res)=\nu(\chi)=1$ means that $\chi$ is trivial on $1+\pi S$ but not on $R^\times$. There are $[S^\times : (1+ \pi S)]=q^2-1$ characters of $S^\times$ that are trivial on $1+\pi S$. Exactly half of them are even, since $-1\notin 1+\pi S$ (residue fields have odd characteristic). The number of characters of $S^\times$ that are trivial on $1+\pi S$ and $R^\times$ (in particular, these characters are even since $-1\in R^\times$) equals the index in $S^\times$ of the subgroup generated by $1+\pi S$ and $R^\times$. As $(1+\pi S)\cap R^\times=1+\pi R$, this index is
\begin{align*}
\frac{[S^\times : (1+ \pi S)]}{[R^\times : (1+ \pi R)]}=\frac{q^2-1}{q-1}=q+1.
\end{align*}
So there are $\tfrac{1}{2}(q^2-1)-(q+1)=\tfrac{1}{2}(q+1)(q-3)$ even characters of $S^\times$ satisfying $\nu(\chi_\res)=\nu(\chi)=1$. This yields a multiplicity of $\tfrac{1}{4}(q+1)(q-3)$ for both $q^{\ell-1/2}$ and $-q^{\ell-1/2}$. Again, when $q=3$, this multiplicity is $0$.

For $k\geq 2$, the condition $\nu(\chi_\res)=\nu(\chi)=k$ means that $\chi$ is trivial on $1+\pi^k S$ but not on $1+\pi^{k-1} R$. A similar count shows that there are
\begin{align*}
&\frac{1}{2}\: \big[S^\times : (1+ \pi^{k-1} S)\big]\Bigg(\big[(1+ \pi^{k-1} S) : (1+ \pi^k S)\big] -\frac{\big[(1+ \pi^{k-1} S) : (1+ \pi^k S)\big]}{\big[(1+ \pi^{k-1} R) : (1+ \pi^k R)\big]}\Bigg)\\
&=\frac{1}{2} (q^2-1)\: q^{2(k-2)}\: (q^2-q)
\end{align*}
 even characters of $S^\times$ satisfying this property. This yields a multiplicity of $\tfrac{1}{4}(q^2-1)\: (q^2-q)\: q^{2(k-2)}$ for both $q^{\ell-k/2}$ and $-q^{\ell-k/2}$.
 
The even characters of $S^\times$ that are trivial on $R^\times$ yield the eigenvalue $- q^{\ell-1}$ when $\nu(\chi)=1$. This eigenvalue has already appeared in previous case, so we have to amend the multiplicity computed before. Here we have to count the non-trivial even characters of $S^\times$ that are trivial on $R^\times$ and on $1+\pi S$. We have essentially seen this count, in the case $k=1$ above. The answer is $(q+1)-1=q$. The readjusted multiplicity of $- q^{\ell-1}$ is thus $\tfrac{1}{4}(q^2-1)\: (q^2-q) + q$.

\emph{Multiplicities for $D^-$}. Again, it suffices to focus on the non-zero egenvalues. The odd characters of $S^\times$ that are non-trivial on $R^\times$ yield eigenvalues $\pm\:\mathrm{i} q^{\ell-k/2}$ when $\nu(\chi_\res)=\nu(\chi)=k$. There are $\tfrac{1}{2}(q^2-1)$ odd characters of $S^\times$ that are trivial on $1+\pi S$. This is the combined multiplicity for $\mathrm{i} q^{\ell-1/2}$ and $-\mathrm{i}q^{\ell-1/2}$. For $k\geq 2$, there are $\frac{1}{2} (q^2-1)\: (q^2-q)\: q^{2(k-2)}$ odd characters of $S^\times$ that are trivial on $1+\pi^k S$ but not on $1+\pi^{k-1} R$. This is the combined multiplicity for $\mathrm{i}q^{\ell-k/2}$ and $-\mathrm{i}q^{\ell-k/2}$.
\end{rem}

Theorem~\ref{thm: platonic R} immediately follows from Lemma~\ref{lem: D+D-} and the remarks preceding it. Now let us consider the case when $R$ is a product $R_1\times\cdots\times R_n$ of finite valuation rings. Then $R^{2,u}$ can be identified with $R_1^{2,u}\times\cdots\times R_n^{2,u}$ and so $D=D_1\otimes\cdots\otimes D_n$. For the even part of $D$, we have
\begin{align*}
D^+=\bigoplus_{\sigma_1\cdots\sigma_n >0} D_1^{\sigma_1}\otimes\cdots\otimes D_n^{\sigma_n}
\end{align*}
and so 
\begin{align*}
sp(D^+)=\bigcup_{\sigma_1\cdots\sigma_n >0} sp(D_1^{\sigma_1})\cdots sp(D_n^{\sigma_n})
\end{align*}
as multisets. This is spectrum of $\mathrm{Pl}(R_1\times\cdots\times R_n)$.

If each factor ring $R_i$ has $q_i\neq 3$ then $sp(D_i^-)\cdot sp(D_{i'}^-)\subseteq sp(D_i^+)\cdot sp(D_{i'}^+)$ for every choice of $i,i'=1,\dots,n$, by Lemma~\ref{lem: D+D-}. Therefore
\begin{align*}
sp(D^+)=sp(D_1^+)\cdots sp(D_n^+)
\end{align*}
as sets. 

If some $R_i$ has $q_i=3$, and $n\geq 2$, then this is no longer true: $sp(D^+)$ strictly contains the set of products $sp(D_1^+)\cdots sp(D_n^+)$. For the sake of notational simplicity, we illustrate this by means of an example. Let $R=R_1\times R_2$, where $q_1=3$ and $q_2\neq 3$. Then $sp(D^+)=sp(D_1^+)\cdot sp(D_2^+)\bigcup sp(D_1^-)\cdot sp(D_2^-)$ contains eigenvalues of the form $\pm \: 3^{\ell_1-1/2}\: q_2^{\ell_2-k/2}$ for $k=1,\dots,\ell_2$. These come from $sp(D_1^-)\cdot sp(D_2^-)$, and they cannot be realized in $sp(D_1^+)\cdot sp(D_2^+)$ since $\pm \: 3^{\ell_1-1/2}$ is missing from $sp(D_1^+)$. Informally, we might say that the missing eigenvalues of $D_1^+$ are only half lost, and they partly re-appear in the product, thanks to $D_1^-$.

What is true, in general, is that the largest, the second largest, and the smallest eigenvalue in $sp(D^+)$ are realized in $sp(D_1^+)\cdots sp(D_n^+)$. In particular, we can read off the extremal non-trivial eigenvalues of the Platonic graph over $\Z/(N)$, as stated in Theorem~\ref{thm: spec PLN}.

\section{Proof of Theorem~\ref{thm: abs E}}\label{sec: technical proof}
In order to understand the proof of Theorem~\ref{thm: abs E}, it will be helpful to start with the simple case of finite fields.
\begin{proof}[Proof of Theorem~\ref{thm: field E}]
Let $\psi$ be an additive character of $\F$, possibly trivial. Denote by $\psi^\ind$ the additive character of $\K$ induced from $\psi$ by pre-composing with the trace. Let also $\chi_\res$ denote the character of $\F^*$ obtained by restricting $\chi$. We may then consider the Gauss sum over $\K$
\begin{align*}
G(\psi^\ind, \chi)=\sum_{s\in \K^*} \psi(\Tr(s))\:\chi(s),
\end{align*}
as well as the Gauss sum over $\F$
\begin{align*}
G(\psi, \chi_\res)=\sum_{c\in \F^*} \psi(c)\:\chi(c).
\end{align*}
These two sums are in fact related. Decomposing over the fibers of the trace map, we write
\begin{align*}
G(\psi^{\mathrm{ind}}, \chi)=\sum_{c\in \F} \psi(c)\: \Big(\sum_{\substack{\Tr(s)=c \\ s\neq 0}} \chi(s)\Big).
\end{align*}
The term corresponding to $c=0$ is $E_0(\chi)$. For $c\neq 0$ we make the change of variable $s\mapsto cs$ in the inner sum, leading to
\begin{align*}
G(\psi^{\mathrm{ind}}, \chi)=E_0(\chi)+\sum_{c\in \F^*} \psi(c) \: \chi(c)\: \Big(\sum_{\Tr(s)=1} \chi(s)\Big)
\end{align*}
that is
\begin{align}\label{eq: ind-res}
G(\psi^{\mathrm{ind}}, \chi)=E_0(\chi)+G(\psi, \chi_{\mathrm{res}})\: E(\chi).
\end{align}
We read \eqref{eq: ind-res} as a linear relation between Eisenstein sums, with Gauss sums as coefficients. At this point, we need to recall some well-known facts about Gauss sums over finite fields. If $\psi$ is an additive character and $\chi$ is a multiplicative character of a finite field $\Bbbk$, then we have the trivial laws
\smallskip
\begin{itemize}[leftmargin=30pt, itemsep=4pt]
\item[(G1)] $G(\psi_0,\chi_0)=|\Bbbk^*|$;
\item[(G2)] $G(\psi_0,\chi)=0$ when $\chi\neq \chi_0$;
\item[(G3)] $G(\psi,\chi_0)=-1$ when $\psi\neq \psi_0$;
\end{itemize}
\smallskip
and, more interestingly, 
\smallskip
\begin{itemize}[leftmargin=30pt, itemsep=4pt]
\item[(G4)] $|G(\psi,\chi)|=\sqrt{|\Bbbk|}$ when $\psi\neq \psi_0$ and $\chi\neq \chi_0$.
\end{itemize}
\smallskip
Returning to the relation \eqref{eq: ind-res}, we consider the following two cases.

\emph{Case 1:} $\chi$ is non-trivial on $\F^*$. Taking $\psi=\psi_0$ in \eqref{eq: ind-res}, we get $E_0(\chi)=0$ by using (G2). Then \eqref{eq: ind-res} becomes 
\begin{align*}
G(\psi^{\mathrm{ind}}, \chi)=G(\psi, \chi_{\mathrm{res}})\: E(\chi).
\end{align*} 
Pick any non-trivial $\psi$, and note that $\psi^{\mathrm{ind}}$ is non-trivial as well, by the surjectivity of the trace. Taking absolute values, and using (G4), we get $|E(\chi)|=q^{(n-1)/2}$.

\emph{Case 2:} $\chi$ is trivial on $\F^*$. Taking $\psi=\psi_0$ in \eqref{eq: ind-res}, and using (G1) and (G2), we obtain $E_0(\chi)=-(q-1)\: E(\chi)$. Thus \eqref{eq: ind-res} turns into 
\begin{align*}
G(\psi^{\mathrm{ind}}, \chi)=-(q-1)\: E(\chi)+G(\psi, \chi_0)\: E(\chi).
\end{align*} 
For $\psi\neq \psi_0$, this says that $G(\psi^{\mathrm{ind}}, \chi)=-q\: E(\chi)$ in light of (G3). Finally, using (G4), we deduce that $|E(\chi)|=q^{n/2-1}$.
\end{proof}

\smallskip

The above proof is our blueprint. We start by extending the relations (G1) - (G4) from finite fields to finite valuation rings. A Gauss sum over a finite valuation ring $R$ has the form
\begin{align*}
G(\psi,\chi)=\sum_{u\in R^\times} \psi(u)\: \chi(u)
\end{align*}
where $\psi$ is a character of the additive group $R$, and $\chi$ is a character of the multiplicative group of units, $R^\times$. The trivial Gauss sums are easily computed.

\begin{lem}\label{lem: trivial Gauss R} The following hold:
\begin{align*}
G(\psi,\chi)= 
\begin{cases}
|R^\times|=(q-1)\: q^{\ell-1} & \textrm{ if } \psi=\psi_0 \textrm{ and } \chi=\chi_0\\
0 & \textrm{ if } \psi=\psi_0 \textrm{ and } \chi\neq \chi_0\\
0 & \textrm{ if } \chi=\chi_0 \textrm{ and } \psi \textrm{ is non-trivial on } (\pi)\\
-|(\pi)|=-q^{\ell-1} & \textrm{ if } \chi=\chi_0 \textrm{ and } \psi \textrm{ is trivial on } (\pi) \textrm{ but } \psi\neq \psi_0
\end{cases}
\end{align*}
\end{lem}

As for multiplicative characters, there is a notion of valuation for additive characters. Given an additive character $\psi$, we write $\nu(\psi)=k$ when $k$ is smallest with the property that $\psi$ is trivial on the additive group $(\pi^k)$. To have valuation $0$ is to be trivial. At the other end of the valuation spectrum, a character - additive or multiplicative - having valuation $\ell$ is said to be \emph{primitive}. For $k\geq 1$, a character has valuation $k$ if and only if it is induced from a primitive character of the ring $R/(\pi^k)$ by pre-composing with the quotient map $R\to R/(\pi^k)$. 

In the next lemma, we compute the absolute value of non-trivial Gauss sums - the analogue of (G4). With some effort, this result can be extracted from Lamprecht's detailed study \cite{Lam}. We prefer to give a direct, self-contained proof, partly based on arguments from \cite[pp.28--30]{BEW} addressing the case $R=\Z/(p^\ell)$.

\begin{lem}\label{lem: abs G}
Let $\psi$ and $\chi$ be non-trivial. Then:
\begin{align*}
|G(\psi,\chi)|= 
\begin{cases}
\sqrt{|R|\: |(\pi^k)|}=q^{\ell-k/2} & \textrm{ if } \nu(\psi)=\nu(\chi)=k\\
0 & \textrm{ if } \nu(\psi)\neq \nu(\chi)
\end{cases}
\end{align*}
\end{lem}

\begin{proof}
\emph{Case 1: different valuations}. Assume $\nu(\chi)=k>\nu(\psi)$. Let $\chi'$ be the multiplicative character of $R/(\pi^k)$ that induces $\chi$. Then:
\begin{align*}
G(\psi,\chi)=\sum_{u\in R^\times} \psi(u)\: \chi'([u])=\sum_{v\in (R/(\pi^k))^\times}\chi'(v) \sum_{\substack{u\in R^\times \\ [u]=v \:}} \psi(u)
\end{align*}
Let $v\in (R/(\pi^k))^\times$, and pick $u_0\in R^\times$ such that $[u_0]=v$. Then $\{u\in R: [u]=v\}=u_0+(\pi^k)$, all of whose elements are actually units in $R$. Therefore
\begin{align*}
\sum_{\substack{u\in R^\times \\ [u]=v \:}} \psi(u)=\sum_{r\in (\pi^k)} \psi(u_0+r)=\psi(u_0)\sum_{r\in (\pi^k)} \psi(r),
\end{align*}
which vanishes since $\psi$ is non-trivial on $(\pi^k)$. We get $G(\psi,\chi)=0$ in this case.

Now assume that $\nu(\psi)=k>\nu(\chi)$. Let $\psi'$ be the additive character of $R/(\pi^k)$ that induces $\psi$. Then:
\begin{align*}
G(\psi,\chi)=\sum_{u\in R^\times} \psi'([u])\: \chi(u)=\sum_{v\in (R/(\pi^k))^\times}\psi'(v) \sum_{\substack{u\in R^\times \\ [u]=v \:}} \chi(u)
\end{align*}
As before, let $v\in (R/(\pi^k))^\times$, and pick $u_0\in R^\times$ such that $[u_0]=v$. We get
\begin{align*}
\sum_{\substack{u\in R^\times \\ [u]=v \:}} \chi(u)=\sum_{r\in (\pi^k)} \chi(u_0+r)=\chi(u_0)\sum_{r\in (\pi^k)} \chi(1+u_0^{-1}r)=\chi(u_0)\sum_{r\in (\pi^k)} \chi(1+r),
\end{align*}
which vanishes since $\chi$ is non-trivial on $1+(\pi^k)$. We get $G(\psi,\chi)=0$ in this case, as well.

\emph{Case 2: equal valuations}. Assume that $\nu(\psi)=\nu(\chi)=k$. Let $\psi'$ and $\chi'$ be the primitive characters of $R/(\pi^k)$ that induce $\psi$, respectively $\chi$. We claim that 
\begin{align*}G(\psi,\chi)=|(\pi^k)|\: G(\psi',\chi')\end{align*} 
the latter Gauss sum being over $R/(\pi^k)$. Indeed, each unit of $R/(\pi^k)$ has $|(\pi^k)|$ lifts to units of $R^\times$, so
\begin{align*}
G(\psi,\chi)=\sum_{u\in R^\times} \psi'([u])\: \chi'([u])=|(\pi^k)|\sum_{v\in (R/(\pi^k))^\times}\psi'(v)\:\chi'(v)=|(\pi^k)|\: G(\psi',\chi').
\end{align*}

With this reduction step at hand, it suffices to prove that $|G(\psi, \chi)|^2=|R|$ whenever $\psi$ and $\chi$ are primitive. We begin by expanding: 
\begin{align*}
\big|G(\psi, \chi)\big|^2&=\sum_{u,v\in R^\times}  \psi(u)\: \chi(u)\:  \overline{\psi}(v)\: \overline{\chi}(v)= \sum_{u,v\in R^\times}  \psi(u-v)\: \chi(u/v)\\
&=\sum_{u,v\in R^\times}  \psi\big((u-1)v\big)\: \chi(u)=\sum_{u\in R^\times}\chi(u) \sum_{v\in R^\times}\psi\big((u-1)v\big)
\end{align*}
The contribution of $u=1$ is $|R^\times|$, and we break up the remainder according to the valuation of $u-1$:
\begin{align*}
\big|G(\psi, \chi)\big|^2&=|R^\times|+\sum_{i=0}^{\ell-1}\sum_{\substack{u\in R^\times\\ \nu(u-1)=i}}\chi(u) \sum_{v\in R^\times}\psi\big((u-1)v\big)
\end{align*}

Let $u\in R^\times$ satisfy $\nu(u-1)=i$, where $0\leq i\leq \ell-1$. Note that the associates $\{(u-1)v: v\in R^\times\}$ represent the set $\{r: \nu(r)=i\}$ with multiplicity $|(\pi^{\ell-i})|$. We have
\begin{align*}
\sum_{v\in R^\times}\psi\big((u-1)v\big)=|(\pi^{\ell-i})|\sum_{\nu(r)=i}\psi(r)=|(\pi^{\ell-i})|\:\Big(\sum_{r\in (\pi^i)} \psi(r)-\sum_{r\in (\pi^{i+1})} \psi(r)\Big)
\end{align*}
which vanishes, unless $i=\ell-1$ in which case it equals $-|(\pi)|$. Here we are using our assumption that $\psi$ is non-trivial on each $(\pi^i)$ for $0\leq i\leq \ell-1$.

Thus
\begin{align*}
\big|G(\psi, \chi)\big|^2= |R^\times|-|(\pi)| \sum_{\substack{u\in R^\times\\ \nu(u-1)=\ell-1}}\chi(u) .
\end{align*}
Now
\begin{align*}
\sum_{\substack{u\in R^\times\\ \nu(u-1)=\ell-1}}\chi(u)=\sum_{\substack{u\in 1+(\pi^{\ell-1})\\ u\neq 1}}\chi(u)=-\chi(1)=-1
\end{align*}
as $\chi$ is non-trivial on $1+(\pi^{\ell-1})$. We conclude that $\big|G(\psi, \chi)\big|^2= |R^\times|+|(\pi)|=|R|$. 
 \end{proof}
 
Finally, we prove Theorem~\ref{thm: abs E} with the help of Lemmas~\ref{lem: trivial Gauss R} and ~\ref{lem: abs G}. 

\begin{proof}[Proof of Theorem~\ref{thm: abs E}]
In addition to $E_0(\chi)$ and $E(\chi)$, consider the higher Eisenstein sums of $\chi$ given by
\begin{align*}
E(\chi, \pi^i)=\sum_{\substack{y\in S^\times \\ \Tr(y)=\pi^i}} \chi(y), \qquad 0\leq i\leq \ell-1
\end{align*}
For $i=0$, we recover $E(\chi)$, while $E_0(\chi)$ could be thought of as corresponding to $i=\ell$.  As in the proof of Theorem~\ref{thm: field E}, we set up linear relations between the Eisenstein sums, with Gauss sums as coefficients. 

Let $\psi$ be an additive character on $R$, and write $\psi^\ind$ for the additive character of $S$ induced by the trace. Note that $\psi^\ind$ has the same valuation as $\psi$, by part i) of Proposition~\ref{prop: standard bilinear form}. Consider $G(\psi^\ind, \chi)$, a Gauss sum over $S$, and write it as follows:
\begin{align*}
G(\psi^\ind, \chi)=\sum_{y\in S^\times} \psi(\Tr(y))\: \chi(y)=\sum_{\substack{y\in S^\times \\ \Tr(y)=0}} \chi(y)+\sum_{i=0}^{\ell-1} \sum_{\substack{y\in S^\times \\ \nu(\Tr(y))=i}} \psi(\Tr(y))\:\chi(y)
\end{align*}
The first term is $E_0(\chi)$. For each $0\leq i\leq \ell-1$ we compute
\begin{align*}
 \sum_{\substack{y\in S^\times \\ \nu(\Tr(y))=i}} \psi(\Tr(y))\:\chi(y)&=\frac{1}{|(\pi^{\ell-i})|}\sum_{u\in R^\times} \sum_{\substack{y\in S^\times \\ \Tr(y)=\pi^iu}} \psi(\Tr(y))\:\chi(y)\\
&=\frac{1}{|(\pi^{\ell-i})|}\sum_{u\in R^\times} \psi(\pi^iu) \sum_{\substack{y\in S^\times \\ \Tr(y)=\pi^iu}} \chi(y)
\end{align*}
which, after a change of variable $y:=yu$ in the right-hand inner sum, becomes
\begin{align*}
 \sum_{\substack{y\in S^\times \\ \nu(\Tr(y))=i}} \psi(\Tr(y))\:\chi(y)&=\frac{1}{|(\pi^{\ell-i})|} \sum_{u\in R^\times} \psi(\pi^iu)\: \chi(u)\: E(\chi, \pi^i)\\
 &= \frac{1}{|(\pi^{\ell-i})|}\: G(\psi_{(i)}, \chi_\res) \: E(\chi, \pi^i)
\end{align*}
where $\psi_{(i)}$ denotes the additive character of $R$ defined by $r\mapsto \psi(\pi^ir)$. Summarizing, we have shown that
\begin{align}\label{eq: general}
G(\psi^\ind, \chi)=E_0(\chi)+\sum_{i=0}^{\ell-1}\frac{G(\psi_{(i)}, \chi_\res)}{|(\pi^{\ell-i})|}\: E(\chi, \pi^i).
\end{align}

\emph{Case 1}: $\chi_\res\neq\chi_0$. Taking $\psi=\psi_0$ in \eqref{eq: general}, we see that each Gauss sum vanishes so we get
\begin{align*}
E_0(\chi)=0.
\end{align*}
Let $\nu(\chi_\res)=k'$, and pick $\psi$ so that $\nu(\psi)=k'$. As $G(\psi_{(i)},\chi_\res)=0$ for $i\neq 0$, we get from \eqref{eq: general} that
\begin{align*}
G(\psi^\ind, \chi)=G(\psi,\chi_\res) \: E(\chi).
\end{align*}
If $k'\neq k$ then $G(\psi^\ind, \chi)=0$ while $G(\psi,\chi_\res)\neq 0$, so $E(\chi)=0$. If $k'= k$ then
\begin{align*}
|E(\chi)|=\frac{|G(\psi^\ind, \chi)|}{|G(\psi,\chi_\res)|}=\frac{(q^n)^{\ell-k/2}}{q^{\ell-k/2}}=q^{(n-1)(\ell-k/2)}.
\end{align*}

\emph{Case 2}: $\chi_\res=\chi_0$. Put 
\begin{align*}
e_i:=-\frac{|R^\times|}{|(\pi^{\ell-i})|}\: E(\chi, \pi^i).
\end{align*}
In particular, $e_0=-|R^\times|\: E(\chi)$.

Taking $\psi=\psi_0$ in \eqref{eq: general}, we have $G(\psi^\ind,\chi)=0$ and $G(\psi_{(i)}, \chi_\res)=G(\psi_0,\chi_0)=|R^\times|$ for each $i$, so we get
\begin{align}\label{eq: last one!}
E_0(\chi)=\sum_{i=0}^{\ell-1} e_i.
\end{align}
Plugging \eqref{eq: last one!} into \eqref{eq: general}, we are led to
\begin{align}\label{eq: general2}
G(\psi^\ind, \chi)=\sum_{i=0}^{\ell-1} \Big(1-\frac{G(\psi_{(i)}, \chi_0)}{|R^\times|}\Big)\: e_i.
\end{align}
Note that the coefficient of $e_i$ is 0 if $\psi_{(i)}$ is trivial; $(1-q^{-1})^{-1}$ if $\psi_{(i)}$ is non-trivial but trivial on $(\pi)$; $1$ if $\psi_{(i)}$ is non-trivial on $(\pi)$. If $\psi$ has valuation $j+1$, then $\psi_{(j+1)}, \dots, \psi_{(\ell-1)}$ are trivial; $\psi_{(j)}$ is non-trivial but trivial on $(\pi)$; $\psi_{(j-1)}, \dots,\psi_{(0)}=\psi$ are non-trivial on $(\pi)$. Thus \eqref{eq: general2} turns into the following relation:
\begin{align}\label{eq: valued}
G(\psi^\ind, \chi)=e_0+\cdots+e_{j-1}+(1-q^{-1})^{-1}\: e_j, \qquad \nu(\psi)=j+1
\end{align}
We will use \eqref{eq: valued} on successive values of $j$.

For $j=0,\dots,k-2$, the Gauss sum in \eqref{eq: valued} vanishes. Inductively, we get
\begin{align*}
e_0=\cdots= e_{k-2}=0.
\end{align*}
In particular, $E(\chi)=0$ when $k\geq 2$.

For $j=k-1$ we get
\begin{align*}
G(\psi^\ind, \chi)= (1-q^{-1})^{-1}\: e_{k-1}.
\end{align*}
As $\nu(\psi^\ind)=\nu(\chi)=k$, we have 
\begin{align*}
|e_{k-1}|= (1-q^{-1})\: |G(\psi^\ind, \chi)|=(1-q^{-1})\: (q^n)^{\ell-k/2}.
\end{align*}
In particular, when $k=1$ we find that $|E(\chi)|=|e_0|/|R^\times|=q^{(n-1)\ell-n/2}$.

For $j=k,\dots,\ell-1$, the Gauss sum in \eqref{eq: valued} vanishes, once again. We inductively get
\begin{align*}
e_{k+s}=-(1-q^{-1})\: q^{-s}\: e_{k-1}, \qquad s=0, \dots,\ell-k-1.
\end{align*}
Thus, by \eqref{eq: last one!},
\begin{align*}
E_0(\chi)&=e_{k-1}+\sum_{s=0}^{\ell-k-1} e_{k+s}=e_{k-1}-(1-q^{-1})\bigg(\sum_{s=0}^{\ell-k-1} q^{-s}\bigg) e_{k-1}\\
&=q^{-(\ell-k)}\: e_{k-1}
\end{align*}
and so we find that
\begin{align*}
|E_0(\chi)|=(1-q^{-1})\: q^{(n-1)\ell -(n/2-1)k}.
\end{align*}
If $k=1$ then $E_0(\chi)=q^{-(\ell-1)}\: e_{0}=-q^{-(\ell-1)}\:|R^\times|\: E(\chi)=-(q-1)\: E(\chi)$.
\end{proof}


\end{document}